%% file: main.tex
\documentclass{article}
\usepackage{graphicx} 

\title{Non-rectifiable Delone sets under pointwise co-Lipschitz bijections}
\author{Ashwin Bhat and Michael Dymond}
\date{\today}
\input{macros}

\begin{document}

\maketitle
\begin{abstract}
    For each $d\in\N_{\geq 2}$ we construct a Delone set $Y$ in $\R^{d}$ for which every Lipschitz bijection from $Y$ to $\Z^{d}$ has a very irregular inverse. For example, the inverse fails to be Lipschitz, even at just a single point. Further, we clarify the relationship between several notions of regularity for bijections between Delone sets, studied in the literature.
\end{abstract}

\section{Introduction}
Delone sets are good models for crystalline structures, which have been of great interest in mathematical literature since Schectman first observed naturally occurring quasicrystals in 1984 \cite{schechtman1}, an achievement that eventually won him the Nobel Prize for Chemistry in 2011.
The study of Delone sets encompasses a wide range of fields beyond quasicrystal theory, including metric geometry, dynamical systems and geometric group theory; it was in the latter two in which Furstenberg (\cite{burago2002rectifying} for more discussion) and Gromov \cite{gromov1992asymptotic} independently asked whether every Delone set in $\R^d$ admits a bi-Lipschitz bijection with with $\Z^d,$ i.e., whether they are \emph{bi-Lipschitz--rectifiable.}  In the late 1990s, McMullen \cite{mcmullen1998lipschitz} and Burago and Kleiner \cite{burago1998separated} showed for each $d \in \N_{\geq 2}$ that Furstenberg and Gromov's question witnesses a negative answer.

Remarkably, this question in discrete geometry was answered by transforming it into a continuous problem. McMullen \cite{mcmullen1998lipschitz} and Burago and Kleiner \cite{burago1998separated} showed that if every Delone set in $\R^d$ is bi-Lipschitz-rectifiable, then every measurable density $\rho:[0,1]^d \to \R_{>0}$ satisfying $0<\inf\rho \leq \sup \rho <\infty$ is the Jacobian of a bi-Lipschitz map $f:[0,1]^d \to \R^d.$ The Jacobian of a bi-Lipschitz map demonstrates local properties which are violated by `checkerboard-like' densities whose average values fluctuate too quickly over small regions. The construction of such densities $\rho$ (see \cite[Chapter $3$]{mcmullen1998lipschitz}, \cite[Lemma $2.1$]{burago1998separated} for examples) hence gives rise to a non-bi-Lipschitz-rectifiable Delone set. In fact, it was shown in \cite{dymond2018mapping} that the set of continuous densities which are the Jacobian of a bi-Lipschitz map is a $\sigma-$porous subset of $C([0,1]^d, \R).$ In this sense, `almost every' positive continuous function gives rise to a Delone set in $\R^d$ non-bi-Lipschitz equivalent to $\Z^d.$

The motivation of the present paper is a question posed by Viera \cite[pg. 3]{viera2023delone}:
\begin{question} \label{question: lipschitz}
    Given $d \in \N$, does every Delone set in $\R^d$ admit a Lipschitz bijection with $\Z^d$?
\end{question}
 When $d = 1,$ it is quick to prove a positive answer to Question \ref{question: lipschitz} since an order-preserving bijection from a Delone set in $\R$ to $\Z$ is always Lipschitz. For each $d \in \N_{\geq 2}$, this question remains open. A main contribution of the present paper, is to provide Delone sets in Euclidean spaces of dimension at least two, admitting no bijections to $\Z^{d}$ belonging to a class of mappings smaller than Lipschitz, but larger than bilipschitz. Specifically we are able to rule out Lipschitz bijections $f$ which are pointwise co-Lipschitz, i.e. having a pointwise Lipschitz inverse. A mapping is pointwise Lipschitz if no single point of its domain witnesses the failure of the Lipschitz condition. Put differently, this means that for any point of its domain the standard quotients for checking the Lipschitz condition involving that point are bounded in size; see Definition~\ref{define: bilip with set}.
 
This follows a building research interest in capturing the extent to which Delone sets may diverge from the integer lattice, by studying weaker forms of rectifiability, based on replacing the bilipschitz condition, with weaker forms of regularity. On the positive side, \cite[Theorem $5.1$]{mcmullen1998lipschitz} shows that every Delone set in $\R^d$ admits a bi-H\"{o}lder homogeneous bijection with $\Z^d$. In comparison, \cite[Theorem $1.2$]{dymond2023highly} shows the existence of a Delone set in $\R^d$ which does not admit a bi-$\omega-$homogeneous bijection with $\Z^d$, where $\omega(t) = t\log(\frac{1}{t})^{\alpha_0(d)}.$
For each $d \in \N_{\geq 2}$, an emerging line of research explores different classes of bijections between Delone sets, with a particular emphasis on weakening the regularity of their inverse. In the present paper, we adopt a proof strategy stemming from \cite{burago1998separated,mcmullen1998lipschitz}, which we call the \emph{pushforward method}: 

Let $\mathcal{K} \subseteq \R^d$ and $\mathcal{F}$ and $\mathcal{G}$ be classes of functions.
\begin{enumerate}[(A)] 
    \item \label{enumerate: assumption before}\label{enumerate: construction before}  Assume every Delone set in $\R^d$ admits an $\mathcal{F}-$bijection onto $\Z^d$.
Then show for each measurable density $\rho:\mathcal{K}\to (0,\infty)$ with $0 < \inf \rho \leq \sup \rho < \infty$ that there exists a $\mathcal{G}-$solution $f:\mathcal{K} \to \R^d$ to 
    \begin{equation} \label{eq: pushforward f(K)}
        \qquad f\#(\rho \leb) = \leb|_{f(\mathcal{K})}.
    \end{equation} 
    \item \label{enumerate: bad rho before} Construct a measurable density $\rho:\mathcal{K} \to (0,\infty)$ with $0 < \inf \rho \leq \sup \rho < \infty$ such that $\rho$ admits no $\mathcal{G}-$solution $f:\mathcal{K} \to \R^d$ to \eqref{eq: pushforward f(K)}.
\end{enumerate}
 The construction in point \ref{enumerate: bad rho before} contradicts the assumption in point \ref{enumerate: assumption before}, so we emerge with a Delone set in $\R^d$ with no $\mathcal{F}-$bijection with $\Z^d.$
 
We highlight that if $f$ is an almost everywhere differentiable homeomorphism, then by the change of variables formula, equation \eqref{eq: pushforward f(K)} is equivalent to the so-called \emph{prescribed Jacobian equation}
\begin{equation} \label{eq: rho jac}
    \rho = |\text{Jac}(f)| \qquad \text{almost everywhere.}
\end{equation} This is particularly relevant when $\mathcal{G}$ is the class of bi-Lipschitz functions. Such a density $\rho$ with no bi-Lipschitz solution $f$ to \eqref{eq: rho jac} was given the name \emph{non-realisable} in \cite{mcmullen1998lipschitz}. However, it is not just in this context where the solvability of \eqref{eq: rho jac} has been of interest: see \cite{dacorogna2007direct,riviere1996resolutions,dymond2018mapping,dymond2023highly} for applications in the fields of mathematical physics, differential equations, metric geometry and functional analysis. 

The reason for using the more general equation \eqref{eq: pushforward f(K)} in the pushforward method is because, in other related literature, the class $\mathcal{G}$ typically does not consist only homeomorphisms. However, by \cite[Lemma~$12.6$]{david1997fractured}, we have that any Lipschitz mapping $f$ satisfying \eqref{eq: pushforward f(K)} is \emph{Lipschitz regular} (see Definition~\ref{define: lip reg}). This crucial fact means, by decomposing $f$ on a suitable open set into finitely many bi-Lipschitz mappings, the authors of \cite{dymond2018mapping} could study the solvability of the pushforward equation \eqref{eq: pushforward f(K)} as an extension of  \cite{burago1998separated,mcmullen1998lipschitz} for the prescribed Jacobian equation \eqref{eq: rho jac}.

All existing detailed accounts of the pushforward method take $\mathcal{K} = [0,1]^d$. The compactness of $[0,1]^d$ allows for point \ref{enumerate: construction before} to be shown by constructing a Delone set which \emph{encodes $\rho$}, a notion introduced in \cite{burago1998separated} and named in \cite[Definition~$1.1$]{bhat2025fast}. Roughly speaking, the idea is to take a disjoint sequence of cubes in $\R^d$ and insert finer and finer `approximations' of $\rho$ -- see \cite{bhat2025fast,dymond2018mapping} for more details. In the current paper we take $\mathcal{K} = \R^d,$ and introduce a new way to encode $\rho:$
\begin{restatable}{define}{encodingdefinition} \label{defn: Y encoding rho}
    Given $d \in \N$, we say that a Delone set $Y \subset \R^d$ \emph{encodes} a measurable density ${\rho\in C_b(\R^d,\R_{>0})}$ if there exists a sequence of positive numbers $(T_n)_{n \in \N} \to \infty$ such that if 
    \begin{equation*} \label{eqn: Y encoding rho}
        \mu_n(A) \coloneqq  \tfrac{1}{T_n^d}{|\tfrac{1}{T_n}Y \cap A|,} \qquad A \subseteq \R^d, \qquad n \in \N
    \end{equation*} then $(\mu_n)_{n \in \N}$ converges vaguely to $\rho \mathcal{L}$ on $\R^d$, where $\leb$ denotes the Lebesgue measure and $\rho\leb$ stands for the measure given by the formula $\rho\leb(A)=\int_{A}\rho\;d\leb$.
\end{restatable} Visually, whenever $d \in \N_{\geq 2}$, a Delone set that encodes $\rho$ in this fashion has finer and finer `approximations' of restrictions of $\rho$ in growing cubes around $0.$

A ubiquitous class of functions in the present paper is that of pointwise co-Lipschitz bijections:
\begin{define} \label{define: bilip with set}
        Let $d \in \N$, $K,X \subseteq \R^d$ and $f:K \to X$. Then $f$ is \emph{Lipschitz at $y \in K$} if there exists $L=L(y)>0$ such that $$||f(x)-f(y)||\leq L||x-y|| \qquad \text{ for every }x \in K.$$ Moreover, we say $f$ is \emph{pointwise Lipschitz} if $f$ is Lipschitz at $y$ for every $y \in K,$ and we say $f$ is \emph{nowhere Lipschitz} if $f$ is not Lipschitz at $y$ for every $y \in K.$ In all of the properties defined for $f$, if the word `Lipschitz' is replaced with `co-Lipschitz', then we mean that $f$ is a bijection and the property holds for $f^{-1}.$
    \end{define}  
    An important classification by Proposition \ref{thm: bilip} is the dichotomy of Lipschitz bijections between Delone sets: such a mapping is either pointwise co-Lipschitz or nowhere co-Lipschitz. We postpone the proof of Proposition \ref{thm: bilip} to the preliminaries section since it is straightforward.
\begin{restatable}{prop}{bilip} \label{thm: bilip}
    Let $d \in \N$, $U \subset \R^d$, $V \subset \R^d$ be a Delone set, $a \in U$ and $f:U \to V$ be a Lipschitz bijection. Then $f$ is co-Lipschitz at $f(a)$ if and only if $f$ is pointwise co-Lipschitz.
\end{restatable}

In the current paper, for a given $d \in \N_{\geq 2}$, we use a modified version of the pushforward method:
\begin{enumerate}[(A')] 
    \item \label{enumerate: assumption} \label{enumerate: pushforward restriction} Assume every Delone set in $\R^d$ admits a Lipschitz, pointwise co-Lipschitz bijection to or from $\Z^d$.
    Then show for each measurable density $\rho:\R^d\to (0,\infty)$ with $0 < \inf \rho \leq \sup \rho < \infty$ that there exists a Lipschitz regular solution $f:\R^d \to \R^d$ to either
    \begin{equation} \label{eq: pushforward on compact sets}
        \qquad f\#(\rho \leb) = \leb,
    \end{equation} 
or 
 \begin{equation} \label{eq: jac inverse pushforward on compact sets}
        \qquad f\#\leb = \rho\leb.
    \end{equation}
    \item \label{enumerate: bad rho} Construct a measurable density $\rho:\R^d \to (0,\infty)$ with $0 < \inf \rho \leq \sup \rho < \infty$ such that every Lipschitz regular $f:\R^d \to \R^d$ solves neither \eqref{eq: pushforward on compact sets} nor \eqref{eq: jac inverse pushforward on compact sets}.
\end{enumerate}

Equation \eqref{eq: jac inverse pushforward on compact sets} is newly introduced in the present paper. Inspired by the framework from \cite{dymond2018mapping}, which shows that the equation $f\#(\rho \leb) = \leb|_{f([0,1]^d)}$ admits a Lipschitz regular solution $f$ only for densities in a $\sigma-$porous subset of $C([0,1]^d,\R)$, we show that \eqref{eq: pushforward on compact sets} and \eqref{eq: jac inverse pushforward on compact sets} respectively witness the same property for a $\sigma-$porous subset of $C_b(\R^d,\R)$. This gives us the existence of a density $\rho$ (in fact, almost any $\rho$) satisfying point \ref{enumerate: bad rho}'.

The contradiction obtained after enacting the modified pushforward method is our main contribution, Theorem \ref{thm: final result 1}.
\begin{restatable}{thm}{finalresult} \label{thm: final result 1}Let $d \in \N_{\geq 2}.$ Then there exists a Delone set in $\R^d$ admitting no Lipschitz, pointwise co-Lipschitz bijection and no pointwise Lipschitz, co-Lipschitz bijection with $\Z^d$, as in Definition \ref{define: bilip with set}.
\end{restatable} 
We emphasise that this is a stronger result than the existence of bilipschitz non-rectifiable Delone sets, from \cite{burago1998separated,mcmullen1998lipschitz}. This is shown in Chapter \ref{chapter 4} where we construct Lipschitz self-bijections of $\Z^d$ that are pointwise co-Lipschitz but not bi-Lipschitz. It is the hope of the authors that thorough exploration of such examples will provide an understanding of the class of nowhere co-Lipschitz mappings between Delone sets, in order to build a framework to answer Question \ref{question: lipschitz} fully. 

A property studied in \cite{viera2023delone} in the context of bijections between Delone sets is \emph{co-uniform continuity}: 
\begin{define}[\cite{viera2023delone}] \label{define: couniform with set}
        Let $d \in \N$ and $K,X \subseteq \R^d$. Given a strictly increasing function ${\omega:(0,\infty) \to (0,\infty)}$, we say that a bijection $f:K \to X$ is
\emph{$\omega-$co-uniformly continuous} if for every $x \in K$ and $r >0$ we have that
\begin{equation*}
    f^{-1}(B(f(x),r)) \subseteq B(x,\omega(r)).
\end{equation*}
Such a bijection $f$ is called \emph{co-uniformly continuous} if it is \emph{$\omega-$}co-uniformly continuous for some $\omega$. Additionally, we say
that a $(\omega-)$co-uniformly continuous map is of \emph{order} $O(h(r))$ if $\omega(r) \in O(h(r))$.
    \end{define}
In the current paper, we show that among Lipschitz bijections between Delone sets, the property of co-uniform continuity is equivalent to bi-Lipschitz continuity. 
A better analogue for uniform continuity of mappings between Delone sets is the notion of \emph{homogeneity}, introduced by \cite{mcmullen1998lipschitz} through the concept of H\"{o}lder homogeneity, and later generalised in \cite{dymond2023highly}:
\begin{define}[\cite{dymond2023highly}, Definition $1.1$] \label{define: co-homogeneity} Let $d\in \N,$ $K, X \subseteq \R^d$ and $\omega:[0,2) \to [0,\infty)$ a non-decreasing function with $\omega(0) = 0$. We say $f:K \to X$ is \emph{$\omega-$homogeneous} if there exists $c>0$ such that for each $R>0$
$$||f(x)-f(y)|| \leq cR\omega\left(\tfrac{||x-y||}{R}\right) \qquad \text{whenever} \; x,y \in B(0,R)\cap K.$$
We say $f$ is \emph{homogeneous} if it is $\omega-$homogeneous for some $\omega.$ Additionally, if $f$ is a bijection, we say $f$ is \emph{$\omega-$co-homogeneous} if $f^{-1}$ is $\omega-$homogeneous, $f$ is \emph{co-homogeneous} if $f^{-1}$ is homogeneous, and the prefix `bi' in place of `co' implies these properties hold for both $f$ and $f^{-1}$.
\end{define}
In the present work, we show that among Lipschitz bijections between Delone sets, the notions of co-homogeneity and pointwise co-Lipschitz continuity are equivalent. 
Therefore, we obtain an equivalent statement to Theorem \ref{thm: final result 1} in Theorem \ref{thm: corollary of PBL delone result}.

\begin{restatable}{thm}{cohom} \label{thm: corollary of PBL delone result}
    Let $d \in \N_{\geq 2}.$ Then there exists a Delone set in $\R^d$ admitting no Lipschitz, co-homogeneous bijection and no co-Lipschitz, homogeneous bijection with $\Z^d,$ as in Definition \ref{define: co-homogeneity}.
\end{restatable}

Among Lipschitz bijections between Delone sets, aided by the equivalences above and examples from the final section of the current paper, we confirm that there are three strictly nested classes of mappings that can be represented in a variety of ways. This is summarised by the implication diagram in Theorem \ref{thm:summary}.
\begin{restatable}{thm}{summary} \label{thm:summary}
    Let $X,Y\subset \R^{d}$ be Delone sets and $f: X\to Y$ be a bijection. Then
\begin{equation} \label{eq: thm summary}
    \begin{split}
        &{\text{$f$ is Lipschitz and} \atop \text{co-uniformly continuous}} \qquad  \;\; \text{$f$ is Lipschitz and co-homogeneous} \\
        &  \;\;\qquad\qquad\Updownarrow \qquad\qquad\qquad\qquad \qquad \qquad\qquad \;\;\;\Updownarrow \\
&\qquad\text{$f$ is bi-Lipschitz}\;\;{\substack{\Longrightarrow\\ \mathrel{\rlap{\hskip .5em/}}\Longleftarrow}}\;\; \;\;\text{$f$ is Lipschitz and pointwise co-Lipschitz} \;\; {\substack{\Longrightarrow\\ \mathrel{\rlap{\hskip .5em/}}\Longleftarrow}} \;\;\text{$f$ is Lipschitz}.\\
    &\qquad\qquad\qquad \qquad\qquad\qquad\qquad\qquad \qquad \qquad\Updownarrow\\
    & \qquad\qquad\qquad \qquad \qquad \qquad \qquad{{\text{$f$ is Lipschitz and $\exists$ $y \in Y$} \atop \text{s.t $f$ is co-Lipschitz at $y$}}}
    \end{split}
\end{equation}
 \end{restatable}   

We highlight that the notion of co-homogeneity has the merit that two functions can be compared by the growth rate of the real-valued functions $\omega$ and $\tilde{\omega}$ for which one is $\omega-$co-homogeneous and the other is $\tilde{\omega}-$co-homogeneous. In the final section, for each non-decreasing and concave function $\omega$, we construct a Lipschitz self-bijection of $\Z^d$ which is $\omega-$co-homogeneous but is not $\tilde{\omega}-$co-homogeneous for any slower growing function $\tilde{\omega}$.

A theorem of McMullen \cite[Theorem $5.1(3)$]{mcmullen1998lipschitz} states that every Delone set in $\R^d$ admits a bi-H\"{o}lder homogeneous bijection with $\Z^d,$ that is, a bijection which is $\omega-$homogeneous and $\omega-$co-homogeneous with $\omega(t) = t^\alpha$, where $\alpha \in (0,1)$.
 By Theorem~\ref{thm: corollary of PBL delone result}, there exists a Delone set $Y \subset \R^d$ such that the bijection from the conclusion of McMullen's theorem is bi-H\"{o}lder homogeneous, but is not $1-$H\"{o}lder homogeneous (i.e., Lipschitz) and is not $1-$co-H\"{o}lder homogeneous (i.e., co-Lipschitz). Since we cannot replace the bi-H\"{o}lder homogeneous condition with a Lipschitz, co-H\"{o}lder homogeneous or a H\"{o}lder homogeneous and co-Lipschitz condition, it is impossible to improve McMullen's result by pushing one of the directions to Lipschitz.  In fact, Theorem~\ref{thm: corollary of PBL delone result} says more: for some Delone sets, if we wish to find a Lipschitz bijection to or from the integer lattice, then it can't have \emph{any} level of co-homogeneity, let alone co-H\"{o}lder homogeneity.
 
\subsection{Structure of the paper}
The key result of Chapter \ref{chapter: viera} is Corollary \ref{cor: couni iff bilip}, which clarifies that the class of Lipschitz co-uniform mappings between two Delone sets, from \cite{viera2023delone}, coincides with the class of bilipschitz mappings. Additionally, we prove Theorem~\ref{thm: corollary of PBL delone result}, showing that co-homogeneity and pointwise co-Lipschitz continuity are equivalent notions among Lipschitz bijections between Delone sets.

The body of the paper focuses on the construction of a Delone set admitting no Lipschitz, pointwise co-Lipschitz bijection with the integer lattice, for which we run the modified pushforward method (\ref{enumerate: assumption}').
Chapter~\ref{chapter: 1.10} is dedicated to the proof of Theorem~\ref{thm: main theorem}, which corresponds to point \ref{enumerate: pushforward restriction}' in the modified pushforward method.
\begin{restatable}{thm}{maintheorem} \label{thm: main theorem}
    Let $d \in \N$. Suppose every Delone set in $\R^d$ admits either a Lipschitz, pointwise co-Lipschitz bijection or a co-Lipschitz, pointwise Lipschitz bijection with $\Z^d$. Then for every measurable density $\rho:\R^d \to \R$ with $0< \inf \rho \leq \sup \rho < \infty,$ there exists a Lipschitz regular map $f:\R^d \to \R^d$ satisfying either 
    \begin{equation*} \label{eq: pushforward equation}
        f\#(\rho\mathcal{L}) = \mathcal{L} \qquad \text{or} \qquad  f\#\leb = \rho\leb. \end{equation*}
\end{restatable}

In Chapter \ref{chapter: 1.11}, we prove Theorem~\ref{thm: pushforward} using ideas from \cite{dymond2018mapping} that show the existence of bounded densities on $[0,1]^d$ admitting no Lipschitz regular solution to (\ref{eq: pushforward on compact sets}, $\mathcal{K} = [0,1]^d$) (so-called \emph{non-realisable densities}). This completes point \ref{enumerate: bad rho}' in the modified pushforward method. In other words, Theorem \ref{thm: pushforward} provides a counterexample to the conclusion of Theorem \ref{thm: main theorem}, so as a corollary we get Theorem \ref{thm: final result 1}, the main result of the current paper. In the spirit of \cite[Theorem~$4.1$]{dymond2018mapping}, we show that, in the sense of porosity, almost any density $\rho \in C_b(\R^d,\R)$ fails to admit a Lipschitz regular solution $f:\R^d \to \R^d$ to either of the pushforward equations from Theorem \ref{eq: pushforward equation}.
\begin{restatable}{thm}{pushforwardthm} \label{thm: pushforward}
    Let $d \in \N_{\geq 2}.$ Then the set of all $\rho\in C_b(\R^d,\R)$ admitting a Lipschitz regular function $f: \R^d \to \R^d$ which solves either of the equations $$f\#(\rho\leb) = \leb \qquad \text{or} \qquad f\#\leb = \rho\leb$$ is a $\sigma-$porous subset of $(C_b(\R^d,\R), ||\cdot||_{\infty}).$
\end{restatable}
In the first section of Chapter \ref{chapter 4} we construct explicit examples of Lipschitz self-bijections of $\Z^d$ which are either nowhere co-Lipschitz or both pointwise co-Lipschitz and non-bi-Lipschitz. The existence of such examples show that the two reverse non-implications in \eqref{eq: thm summary} of Theorem \ref{thm:summary} hold.
Hence, Theorem \ref{thm: final result 1} is on one hand stronger result than the classical result of Burago and Kleiner \cite[Theorem~$1.1$]{burago1998separated}, but on the other hand does not fully answer Question \ref{question: lipschitz}. 
We finish this section by further studying what level of co-homogeneity each of the Lipschitz self-bijections of $\Z^d$ witness.

\subsection{Notation and preliminaries}
\paragraph{Sets and measures}
We use $||\cdot||$ and $||\cdot||_\infty$ to denote the Euclidean norm and the supremum norm respectively. In a normed space $(M,||\cdot||_M),$ if $r>0$ and $x \in K,$ we write $B_{(M,||\cdot||_M)}(x,r)$ to denote the open ball, and $\bar{B}_{(M,||\cdot||_M)}(x,r)$ to denote the closed ball, of radius $r$ centred at $x$ with respect to the norm $||\cdot||_M$. If $M$ is not specified, it is implied that $M=\R^d$, and if $||\cdot||_M$ is not specified, it represents the implied norm on $M$.

We refer to the set of all continuous functions from $M$ to $\R$ as $C(M,\R)$, and its subsets of bounded functions and functions with compact support are denoted by $C_b(M, \R)$ and $C_c(M,\R)$ respectively. We refer to the support of a function $f:M \to \R$ as $\text{supp}(f),$ that is
$$\text{supp}(f) \coloneqq \overline{\{x \in M: f(x) \neq 0\}}.$$
For $n \in \mathbb{N},$ we write $[n]$ to denote the set $\{1,...,n\}.$ 
We say a property on a measure space holds \emph{almost everywhere} or \emph{a.e.} if it occurs everywhere on its domain except on a set of measure zero. 
Given a function $g:\R^d\to \R^m$ and measure $\mu$ on $\R^d,$ the pushforward measure $g \# \mu$ on $\R^m$ is defined as
\begin{equation*}
    g \# \mu(A) \coloneqq \mu(g^{-1}(A)), \qquad A \subset \R^m.
\end{equation*}
We write $\mathcal{L}$ to denote the $d-$dimensional Lebesgue measure on $\mathbb{R}^d.$ For a measurable function $\rho\in C_b(\R^d,\R_{>0}),$ the measure $\rho\mathcal{L}$ is defined as
\begin{equation*}
    \rho \mathcal{L}(A) \coloneqq \int_{A}\rho\; d\mathcal{L}, \qquad A \subset \R^d.
\end{equation*} 
Given a measure $\mu$ and a sequence of measures $(\mu_n)_{n \in \N}$ on $M \subseteq \R^d$, suppose for all $f \in \mathcal{F}$ that
\begin{equation*}
    \bigg|\int_M f \; d\mu_n -\int_M f \; d\mu\bigg| \to 0 \;\; \text{as}  \; n \to \infty.
\end{equation*}Then we say $(\mu_n)_{n \in \N}$ \emph{converges weakly} to $\mu$ if $\mathcal{F} = C_b(M,\R),$ and $\mu_n$ \emph{converges vaguely} to $\mu$ if $\mathcal{F} = C_c(M, \R).$

\begin{prop} \label{prop: vague weak prop} Let $d \in \N$ and let $(\nu_n)_{n \in \N}$ and $\nu$ be Borel measures on $\R^d.$ If for each compact set $\mathcal{K} \subset \R^d$ we have that $(\nu_n|_\mathcal{K})_{n \in \N}$ converges weakly to $\nu|_{\mathcal{K}},$ then
     $(\nu_n)_{n \in \N}$ converges vaguely to $\nu$. 
     
\end{prop}
\begin{proof}
  Let $\Psi \in C_c(\R^d,\R).$ Let $\mathcal{K} \coloneqq \text{supp}(\Psi).$  Since $(\nu_n|_\mathcal{K})_{n \in \N}$ converges weakly to $\nu|_\mathcal{K}$
    \begin{equation*}
        \begin{split}
        \lim_{n \to \infty}\left(\int_{\R^d}\Psi\; d\nu_n - \int_{\R^d}\Psi \; d\nu\right) 
            = \lim_{n \to \infty}\left(\int_{\mathcal{K}} \Psi|_{\mathcal{K}}\;d(\nu_n|_{\mathcal{K}}) - \int_{\mathcal{K}} \Psi|_{\mathcal{K}}\;d(\nu|_{\mathcal{K}}) \right) = 0,
        \end{split}
    \end{equation*} so we conclude that $(\nu_n)_{n \in \N}$ converges vaguely to $\nu.$
\end{proof}

\begin{define}[\cite{lagarias2003repetitive}, Definition $1.1$] \label{define: delone} Let $d \in \N$. A \emph{separated net}, also known as a \emph{Delone set}, is a discrete set $Y \subset \R^d$ satisfying the following two criteria:
\begin{itemize}
    \item (Uniform discreteness) There exists an $s>0$ such that every closed ball of radius $s$ in $\R^d$ contains at most one point of $Y,$ i.e., $Y$ is $2s-$separated. We call the biggest such $s$ the \emph{packing radius} of $Y.$ 
    \item (Relative density) There exists an $S>0$ such that every closed ball of radius $S$ in $\R^d$ contains at least one point of $Y.$ We call the smallest such $S$ the \emph{covering radius} of $Y.$
\end{itemize}
\end{define}

\paragraph{Mappings}
Given $d \in \N$, $K, X \subseteq \R^d$ and $L>0,$ we say a function $f:K \to X$ is \emph{($L$-)Lipschitz} if for every pair $x,y \in K$ $$||f(x)-f(y)|| \leq L||x-y||.$$ The smallest such $L$ is called the \emph{Lipschitz constant} of $f$ and is denoted $\text{Lip}(f).$ Additionally, if $f$ is a bijection, we say $f$ is \emph{co-Lipschitz} if $f^{-1}$ is Lipschitz, and if $L_1,L_2>0$, we say $f$ is \emph{$(\frac{1}{L_1},L_2-)$bi-Lipschitz} if $f$ is $(L_2-)$Lipschitz and $f^{-1}$ is $(L_1-)$Lipschitz. 

Given a set $\mathcal{K} \subseteq \R^d$, we refer to a measurable real-valued function $\rho:\mathcal{K} \rightarrow \mathbb{R}_{>0}$ as a \textit{density}. 

For non-negative real-valued functions $f,g:\R^d \to \R_{\geq 0}$, we write
    $f(x) \in o(g(x))$ if $\lim_{x \to \infty} \frac{f(x)}{g(x)} = 0,$ and $f(x) \in O(g(x))$ if there exists a constant $C>0$ such that $f(x) \leq Cg(x)$ for each $x \in \R_{\geq0}.$
    \begin{define}[\cite{viera2023delone}, Definition $3.1$] \label{define: lip reg}
    Let $d \in \N$, $K,X \subseteq \R^d$ and $C \in \N.$ We say that a Lipschitz map $f:K \to X$ is \emph{Lipschitz regular} if for every ball $B \subset X$ of radius $R > 0$, the set $f^{-1}(B)$ can
be covered by at most $C$ balls of radius $CR$. The smallest value of $C$ that works for every $R>0$ is called
the \emph{regularity} of $f$, and is denoted by \emph{Reg}$(f)$. 
\end{define}
    \bilip*
\begin{proof}
   It is clear that if $f$ is pointwise co-Lipschitz, then $f$ is co-Lipschitz at $f(a)$. Conversely, using the bijectivity of $f$, let $b \in U$ be such that $0<||f(b)-f(a)|| \leq 2S,$ where $S=S(V)>0$ is the covering radius of $V.$ We will show $f$ is co-Lipschitz at $f(b)$, which when applied inductively will show $f$ is co-Lipschitz at every point of $V$.
   
   Since $f$ is co-Lipschitz at $f(a),$ there exists $\beta>0$ such that $$||f(y)-f(a)|| \geq \beta||y-a||\qquad \text{ for every } y \in U .$$ Now let $x \in U\setminus \{b\}$. If $||x-a||\geq \frac{4S}{\beta},$ then we have 
    \begin{equation*} \label{eq: bilip estimate}
        \begin{split}
            \frac{||f(x)-f(b)||}{||x-b||} &\geq \frac{||f(x)-f(a)||- ||f(a)-f(b)||}{||x-a||+||a-b||}=\frac{\frac{||f(x)-f(a)||}{||x-a||}- \frac{||f(a)-f(b)||}{||x-a||}}{1+\frac{||a-b||}{||x-a||}} \\&\geq \frac{\beta-\frac{2S}{||x-a||}}{1+{\frac{||a-b||}{||x-a||}}}\geq \frac{\frac{\beta}{2}}{1+\frac{\beta||a-b||}{4S}} = \frac{2\beta S}{4S+\beta||a-b||}.
        \end{split}
    \end{equation*}
Otherwise, $||x-a|| < \frac{4S}{\beta}.$ There are finitely many points $x$ satisfying this inequality. Indeed, since $f$ is a Lipschitz bijection, we have that
$$B\left(a,\tfrac{4S}{\beta}\right) \cap U \subseteq f^{-1}\left(B\left(f(a),\tfrac{4S\text{Lip}(f)}{\beta}\right)\cap V\right),$$ and the right-hand set contains finitely many points since $V$ is a Delone set and is thus uniformly discrete.
Therefore, the quantity $$\delta \coloneqq \min_{x \in (B(a, \frac{4S}{\beta})\setminus \{b\})\cap U} \frac{||f(x)-f(b)||}{||x-b||}$$ is well-defined, and moreover is positive since $f$ is a bijection. Thus, for each $x \in U$
$$\frac{||f(x)-f(b)||}{||x-b||} \geq \min\bigg\{\delta, \tfrac{2\beta S}{4S+\beta||a-b||}\bigg\}>0,$$
so $f$ is co-Lipschitz at $f(b)$.
\end{proof}

\section{Equivalent notions for Lipschitz bijections between Delone sets} \label{chapter: viera}
\subsection{Co-uniform continuity is equivalent to bi-Lipschitz continuity}
Restrictions of Question \ref{question: lipschitz} for Lipschitz, co-uniformly continuous bijections (see Definition \ref{define: couniform with set}) are studied in \cite{viera2023delone}. We show, using the `stepping-stone' property of Delone sets attributed to Lagarias \cite[Lemma $2.2$]{lagarias1999geometric} that in the context of Lipschitz mappings between Delone sets, co-uniform continuity is in fact equivalent to bi-Lipschitz continuity.

\begin{prop} \label{prop: iff holder couniformity} Let $d \in \N$, $\omega:(0,\infty) \to (0,\infty)$ be a strictly increasing function, $U, V \subseteq \R^d$ with $\diam U = \infty$, and $f:U\to V$ be a bijection. Then $f$ is $\omega$-co-uniformly continuous if and only if
\begin{equation}\label{eq: omega inverse}
    ||f^{-1}(u)-f^{-1}(v)|| \leq \omega(||u-v||) \qquad \text{for all $u,v \in V$}.
\end{equation}
\end{prop} 
\begin{proof}
     Let $r>0$ and $x \in U$ be given and suppose \eqref{eq: omega inverse} holds. Now if $y \in f^{-1}(B(f(x),r)),$ then
     $$\omega(r)> \omega(||f(x)-f(y)||) \geq ||x-y||,$$
     so $y \in B(x, \omega(r)).$ Thus, $f^{-1}(B(f(x),r)) \subseteq B(x,\omega(r)),$ so $f$ is $\omega-$co-uniformly continuous.

    Conversely, suppose $f$ is $\omega-$co-uniformly continuous. 
    First, we claim that $\sup \omega = \infty.$ Indeed, suppose for a contradiction that $\sup \omega < \infty.$ Notice that since $\diam U = \infty,$ there exist $x,y \in U$ such that $||x-y|| > \sup \omega.$ Set $r = 2||f(x)-f(y)||,$ which is positive since $f$ is bijective and $x \neq y.$ Then since $f$ is $\omega-$co-uniformly continuous and $||f(x)-f(y)|| <r,$ we have that $||x-y|| < \omega(r) \leq \sup \omega,$ which is a contradiction. 
    
   Since $\omega$ is a strictly increasing function and $\sup \omega = \infty$, $\omega$ has a well-defined and strictly increasing inverse ${\omega^{-1}:(0, \infty) \to (0, \infty)}$. It is clear \eqref{eq: omega inverse} holds when $u = v.$
    Now fix $u \in V$ and $v \in V\setminus \{u\}.$ Take $r = \omega^{-1}(||f^{-1}(u)-f^{-1}(v)||)>0$. Then since $f^{-1}(B(u,r)) \subseteq B(f^{-1}(u), \omega(r))$ $$||u-v|| \geq r = \omega^{-1}(||f^{-1}(u)-f^{-1}(v)||),$$ so \eqref{eq: omega inverse} holds.
\end{proof}
The following lemma shines a light on how the structural properties of a Delone set highly restrict the nature of the so-called \textit{$\omega-$uniform continuity} (see Lemma \ref{lemma: couni iff bilip}\eqref{point 1: couni iff bilip}) of mappings endowed on it. In particular, a mapping from a Delone set that controls the distance of image points in an $\omega-$uniformly continuous manner (for example, $\omega(r) = r^2$) always collapses into a Lipschitz mapping (corresponding to $\omega(r) \in O(r))$.
\begin{lemma} \label{lemma: couni iff bilip}
    Let $d \in \N,$ $V\subset \R^d$ a Delone set and $F:V \to \R^d.$ Then the following are equivalent: 
    \begin{enumerate}[(i)]
        \item \label{point 1: couni iff bilip} There exists a strictly increasing function $\omega:(0,\infty) \to (0,\infty)$ such that
        $$||F(x)-F(y)|| \leq \omega(||x-y||) \qquad \text{for all $x,y \in V$};$$
        \item \label{point 2: couni iff bilip} There exists a strictly increasing function $\omega: (0,\infty) \to (0,\infty)$ with $\omega(r) \in O(r)$ and
        $$||F(x)-F(y)|| \leq \omega(||x-y||) \qquad \text{for all $x,y \in V$};$$
        \item  \label{point 3: couni iff bilip} $F$ is Lipschitz.
    \end{enumerate}
\end{lemma}
\begin{proof}
 $(\eqref{point 2: couni iff bilip} \Rightarrow \eqref{point 1: couni iff bilip})$ is clear.  For $(\eqref{point 3: couni iff bilip}\Rightarrow \eqref{point 2: couni iff bilip})$, take $\omega(r) \coloneqq \text{Lip}(F)r \in O(r).$
 For $(\eqref{point 1: couni iff bilip} \Rightarrow \eqref{point 3: couni iff bilip})$, let $s,S>0$ be the packing and covering radius of $V$, respectively. Fix $x,y \in V.$ By the `stepping-stone' property of Delone sets (\cite[Lemma $2.2$]{lagarias1999geometric}), there exists a chain of points $x = x_0,x_1,...,x_m = y$ in $V$ with $m \leq \frac{4S}{s^2}||x-y||$ and $||x_i-x_{i-1}||\leq 2S$ for each $i \in [m].$ Thus, 
    $$||F(x)-F(y)|| \leq \sum_{i = 1}^m ||F(x_i)-F(x_{i-1})|| \leq \sum_{i=1}^m \omega(||x_i - x_{i-1}||) \leq \sum_{i=1}^m \omega(2S) \leq \omega(2S)\frac{4S}{s^2}||x-y||,$$
 so $F$ is Lipschitz.
 \end{proof}
\begin{restatable}{cor}{vieracorollary} \label{viera corollary}
    \label{cor: couni iff bilip}
    Let $d \in \N,$ $U \subset \R^d$ with $\diam U = \infty$, $V\subset \R^d$ a Delone set, and $f:U \to V$ a bijection. Then the following are equivalent: 
    \begin{enumerate}[(i)]
        \item  \label{point 1: cor} $f$ is co-uniformly continuous;
        \item  \label{point 2: cor} $f$ is co-uniformly continuous of order $O(r)$;
        \item \label{point 3: cor} $f^{-1}$ is Lipschitz.
    \end{enumerate} 
\end{restatable}
   \begin{proof}
In light of Proposition~\ref{prop: iff holder couniformity}, the present corollary is exactly the statement of Lemma~\ref{lemma: couni iff bilip} for $F=f^{-1}$.
   \end{proof}

   \subsection{Co-homogeneity is equivalent to pointwise co-Lipschitz continuity}
In this section, we show the equivalence of pointwise co-Lipschitz continuity and co-homogeneity in the class of Lipschitz bijections between Delone sets. We refer the reader to Definitions \ref{define: bilip with set} and \ref{define: co-homogeneity} for definitions of these two notions. This allows for Theorem~\ref{thm: final result 1} to be equivalently stated as Theorem~\ref{thm: corollary of PBL delone result}. In Chapter \ref{chapter 4}, we construct non-bilipschitz examples of Lipschitz, pointwise co-Lipschitz bijections between Delone sets, and so these notions define a strictly larger class of of bijections between Delone sets than the class of bilipschitz bijections.
\begin{thm} \label{thm: co-hom equivalent to pointwise BL}
    Let $d \in \N$, $K \subset \R^d$ be a uniformly discrete set, $X \subseteq \R^d$, $a \in K$ and $f:K \to X$ be a mapping. Then $f$ is Lipschitz at $a$ if and only if $f$ is homogeneous.
\end{thm}
\begin{proof} If $K = \{a\},$ then $f$ is trivially homogeneous and Lipschitz at $a$. For the remainder of the proof, we will assume $K$ contains at least two distinct points. Given $K$ contains at least two distinct points and is uniformly discrete, we have that $p \coloneqq \inf_{z \in K \setminus \{0\}} ||z||$ is well-defined and positive, and further, there exists $s>0$ such that $||x-y||\geq s$ for every pair $x,y \in K.$

    First suppose $f$ is homogeneous. Then there exists a non-decreasing function $\omega:[0,2) \to [0,\infty)$ and $c>0$ such that for each $R>0$ 
    \begin{equation} \label{eq: homogeneity}
        ||f(x)-f(y)|| \leq cR \omega\left(\tfrac{||x-y||}{R}\right) \qquad \text{whenever} \; x,y \in B(0,R) \cap K.
    \end{equation}
Fix $x \in K\setminus \{a\}.$ Take $R \coloneqq 2(||x-a||+||a||)>0.$ Notice $R >||a||$ and 
$$||x|| \leq ||x-a|| + ||a|| < R,$$ so $x,a \in B(0,R).$ Then by \eqref{eq: homogeneity}
\begin{equation*}
    \begin{split}
        ||f(x)-f(a)|| &\leq 2c(||x-a||+||a||) \omega\left(\tfrac{||x-a||}{2(||x-a||+||a||)}\right) \leq  2c\omega\left(\tfrac{1}{2}\right)(||x-a||+||a||)
        \\&= 2c\omega\left(\tfrac{1}{2}\right)\left(1+\tfrac{||a||}{||x-a||}\right)||x-a|| \leq 2c\omega\left(\tfrac{1}{2}\right)\left(1+\tfrac{||a||}{s}\right)||x-a||,
    \end{split}
\end{equation*}
so $f$ is Lipschitz at $a$.

Conversely, suppose $f$ is Lipschitz at $a.$ Then there exists $\tilde{L}>0$ such that
\begin{equation} \label{eq: tilde L}
   ||f(x)-f(a)||\leq \tilde{L}||x-a|| \qquad \text{whenever }x \in K.
\end{equation}
By \eqref{eq: tilde L}, taking $L= L(a) \coloneqq \tilde{L}\left(1+\tfrac{||a||}{p}\right),$ then for each $x \in K \setminus \{0\}$
\begin{equation} \label{eq: L}
    ||f(x)-f(a)|| \leq \tilde{L}||x-a||\leq \tilde{L}(||x|| + ||a||) = \tilde{L}\left(1+\tfrac{||a||}{||x||}\right)||x|| \leq \tilde{L}\left(1+\tfrac{||a||}{p}\right)||x|| = L||x||.
\end{equation}
Define the function $\omega:[0,2) \to [0,\infty)$ by
$$\omega(t) \coloneqq \sup_{S>0} \sup_{\substack{u,v \in B(0,S) \cap K\\||u-v|| \leq tS}} \frac{||f(u)-f(v)||}{S}.$$
We claim that $f$ is $\omega-$homogeneous. To this end, we must check the conditions of Definition \ref{define: co-homogeneity} hold for $\omega$. Clearly, we have that $\omega(0) = 0$. Next, we claim $\omega$ is a finite-valued function. Indeed, suppose $S >0$ and $u,v \in B(0,S)\cap K$. Without loss of generality, if $u = 0$ and $v \neq 0$, then we see by \eqref{eq: tilde L} that $||f(u)-f(a)|| = ||f(0)-f(a)|| \leq \tilde{L}||a||,$ so since $S \geq ||v|| >0$
$$ \tfrac{1}{S}{||f(u)-f(v)||} \leq \tfrac{1}{||v||}{||f(u)-f(a)||}+\tfrac{1}{||v||}{||f(v)-f(a)||} \leq \tfrac{1}{p}{\tilde{L}||a||} + L= (L-\tilde{L})+L\leq 2L.$$
Now consider $u, v \neq 0.$
Notice since $S \geq ||u||>0$ and $S \geq ||v||>0$ that
\begin{equation*}
     \tfrac{1}{S}{||f(u)-f(v)||} \leq \tfrac{1}{S}{||f(u)-f(a)||}+\tfrac{1}{S}{||f(v)-f(a)||} \leq \tfrac{1}{||u||}{||f(u)-f(a)||}+\tfrac{1}{||v||}{||f(v)-f(a)||} \leq 2L.
\end{equation*}
    Therefore, for each $t \in [0,2)$ $$\omega(t) \coloneqq \sup_{S>0} \sup_{\substack{u,v \in B(0,S) \cap K\\||u-v|| \leq tS}}  \frac{||f(u)-f(v)||}{S} \leq 2L.$$
    Finally, we claim $\omega$ is a non-decreasing function. Indeed, given $t, t' \in (0,2)$ with $t \leq t',$ notice for each $S >0$ that $$\{u,v \in B(0,S)\cap K: ||u-v|| \leq tS\} \subseteq \{u,v \in B(0,S)\cap K: ||u-v|| \leq t'S\},$$ so
    $$\sup_{\substack{u,v \in B(0,S)\cap K\\||u-v|| \leq tS}} \frac{||f(u)-f(v)||}{S} \leq  \sup_{\substack{u,v \in B(0,S)\cap K\\||u-v|| \leq t'S}} \frac{||f(u)-f(v)||}{S},$$
    so $\omega(t) \leq \omega(t'),$ and hence $\omega$ is non-decreasing.
     We conclude $\omega$ satisfies all the conditions of Definition \ref{define: co-homogeneity}.
    
   Now fix $R>0$ and $x,y \in B(0,R)\cap K.$ Then
    \begin{equation*}
        \omega\left(\frac{||x-y||}{R}\right) \ = \sup_{S>0} \sup_{\substack{u,v \in B(0,S)\cap K\\||u-v|| \leq \frac{||x-y||}{R}S}} \frac{||f(u)-f(v)||}{S} \geq \sup_{\substack{u,v \in B(0,R)\cap K\\||u-v|| \leq ||x-y||}} \frac{||f(u)-f(v)||}{R}\geq \frac{||f(x)-f(y)||}{R},
    \end{equation*} so $f$ is $\omega-$homogeneous.
\end{proof}
As a consequence of Theorem~\ref{thm: co-hom equivalent to pointwise BL}, we get that the pointwise co-Lipschitz property is the same as co-homogeneity for Lipschitz bijections to a Delone set. Thus, we recognise that Theorem~\ref{thm: corollary of PBL delone result} is a restatement of Theorem~\ref{thm: final result 1}:
\begin{cor} \label{cor: cohom equivalence}
    Let $d \in \N,$ $U \subset \R^d$, $V \subset \R^d$ be a Delone set, and $f:U \to V$ be a Lipschitz bijection. Then $f$ is pointwise co-Lipschitz if and only if $f$ is co-homogeneous. 
\end{cor}
\begin{proof} Fix a point $v \in V.$ By Proposition \ref{thm: bilip}, $f$ is co-Lipschitz at $v$ if and only if $f$ is pointwise co-Lipschitz. Now apply Theorem \ref{thm: co-hom equivalent to pointwise BL} to $f^{-1}$ to obtain that $f$ is co-Lipschitz at $v$ if and only if $f$ is co-homogeneous.
\end{proof}

\begin{cor}\label{cor:two_theorems_the_same}
    Theorem~\ref{thm: final result 1} and Theorem~\ref{thm: corollary of PBL delone result} are equivalent.
\end{cor}
\begin{proof}
 This is immediate from Corollary~\ref{cor: cohom equivalence}.
\end{proof}
 
\section{Proof of Theorem \ref{thm: main theorem}} \label{chapter: 1.10}
In this section, we will lay the groundwork to prove Theorem \ref{thm: main theorem}. There are crucial details in this section which differ from prior literature using the pushforward method (\ref{enumerate: assumption before}). This stems from the lack of compactness of $\R^d$, which is the domain of both the density $\rho$ to be encoded in a Delone set, as well as of a sequence of Lipschitz mappings we will construct later. To this end, we consider vague convergence instead of weak convergence for sequences of measures, and uniform convergence on bounded sets rather than global uniform convergence for sequences of Lipschitz mappings on $\R^d$.
\begin{lemma}[\cite{dymond2018mapping}, variant of Lemma $5.5$] \label{lemma: pushforward measures vaguely converging} Let $d \in \N$, and $(\nu_n)_{n \in \N}$ and $\nu$ be Borel measures on $\R^d$ such that $(\nu_n)_{n \in \N}$ converges vaguely to $\nu$. Assume for each bounded set $\mathcal{K} \subset \R^d$ there exists $c>0$ such that 
\begin{equation} \label{eq: lemma condition 1}
    \nu_n(\mathcal{K}) \leq c \qquad \text{for each $n \in \N$}.
\end{equation}
Let $h_n:\R^d \to \R^d$ be a sequence of continuous mappings converging uniformly to $h:\R^d\to \R^d$ on every bounded set. Assume for each $r>0$ there exists $R>0$ with
\begin{equation} \label{eq: lemma condition 2}
    h_n^{-1}(B(0,r))\cup h^{-1}(B(0,r)) \subseteq B(0,R) \qquad \text{ for each $n \in \N$}.
\end{equation}Then $(h_n \# \nu_n)_{{n\in\N}}$ converges vaguely to $h \# \nu$.
    \begin{proof}
        Fix $\Psi:\R^d\to \R$ which is continuous with compact support. In particular, there exists $r >0$ such that $\text{supp}(\Psi)\subseteq B(0,r).$ Then writing $S_n \coloneqq \text{supp}(\Psi \circ h_n) \cup \text{supp}(\Psi \circ h)$ for each $n \in \N$ as shorthand
        \begin{equation} \label{eq: pushforward lemma}
            \begin{split}
        &\bigg|\int_{h_n(\R^d)}\Psi \; d(h_n \# \nu_n) - \int_{h(\R^d)}\Psi \; d(h \# \nu)\bigg| = \bigg|\int_{\R^d}\Psi\circ h_n \; d\nu_n - \int_{\R^d}\Psi \circ h\; d\nu\bigg| 
        \\&\leq\int_{\R^d}|\Psi\circ h_n - \Psi \circ h|\; d\nu_n+\bigg|\int_{\R^d}\Psi\circ h \; d\nu_n - \int_{\R^d}\Psi \circ h\; d\nu\bigg|\\
        &\leq \nu_n(S_n)\;\sup_{x \in S_n}|\Psi\circ h_n(x) - \Psi \circ h(x)|+ \bigg|\int_{\R^d}\Psi\circ h \; d\nu_n - \int_{\R^d}\Psi \circ h\; d\nu\bigg|.
            \end{split}
        \end{equation} By \eqref{eq: lemma condition 2}, there exists $R>0$ such that for each $n \in \N$ we have
        $$S_n\subseteq h_n^{-1}(\text{supp}(\Psi)) \cup h^{-1}(\text{supp}(\Psi)) \subseteq h_n^{-1}(B(0,r)) \cup h^{-1}(B(0,r))\subseteq B(0,R),$$
        so by \eqref{eq: lemma condition 1}, there exists $c>0$ such that $\nu_n(S_n) \leq c$ for each $n \in \N$. 
        Consequently, the first term of the last line of \eqref{eq: pushforward lemma} is bounded by $$c\sup_{x \in B(0,R)} |\Psi \circ h_n(x) - \Psi \circ h(x)|,$$ which goes to $0$ as $n \to \infty$ since $(\Psi \circ h_n)_{n \in \N}$ converges uniformly to $\Psi \circ h$ on $B(0,R)$. The second term of the last line of \eqref{eq: pushforward lemma} goes to $0$ as $n \to \infty$ since $(\nu_n)_{n \in \N}$ converges vaguely to $\nu$ and $\Psi \circ h$ is a continuous function with support contained in $B(0,R)$. 
    \end{proof}
\end{lemma}
\begin{lemma} \label{lemma: lip extension on Rd}
    Let $d \in \N$, $L>0$ and $h_n:\R^d\to \R^d$ be a sequence of $L-$Lipschitz maps such that $||h_n(0)|| \xrightarrow{n \to \infty} 0.$ Then there exists an $L-$Lipschitz map $h:\R^d \to \R^d$ and a strictly increasing sequence $(m_k)_{k \in \N} \subseteq \N$ such that for each bounded set $\mathcal{K} \subset \R^d,$ $(h_{m_k}|_\mathcal{K})_{k \in \N}$ converges uniformly to $h|_\mathcal{K}$.
\end{lemma}
\begin{proof}
Since $||h_n(0)|| \xrightarrow{n \to \infty} 0$, there exists $C>0$ such that $||h_n(0)|| \leq C$ for each $n \in \N.$ Therefore, given $j \in \N,$ we have for each $n \in \N$ and $x \in \overline{B}(0,j)$ that
    $$||h_n(x)|| \leq ||h_n(x)-h_n(0)||+||h_n(0)||\leq L||x||+C\leq Lj+C,$$ so $(h_n|_{\overline{B}(0,j)})_{n \in \N}$ is uniformly bounded. 

Inductively, we construct a sequence of subsequences of $\N,$ $((n_k^{(j)})_{k \in \N})_{j \in \N},$ as follows. Set $(n_k^{(0)})_{k \in \N} \coloneqq \N,$ and let $j \in \N.$
Since the sequence $(h_n|_{\overline{B}(0,j)})_{n \in \N}$ is uniformly bounded and uniformly $L-$Lipschitz, then by the Arzelà-Ascoli theorem we may take a subsequence $(n_k^{(j)})_{k \in \N} \subseteq (n_k^{(j-1)})_{k \in \N}$ such that $(h_{n_k^{(j)}}|_{\overline{B}(0,j)})_{k \in \N}$ is uniformly convergent. 

Now define the diagonal sequence $(m_k)_{k \in \N}\subseteq \N$ given by $m_k \coloneqq n_k^{(k)}$ for each $k \in \N.$ It is clear that $(m_k)_{k \in \N}$ is strictly increasing since 
$$m_{k+1} =n^{(k+1)}_{k+1} \geq n_{k+1}^{(k)}>n_{k}^{(k)}=m_k, \qquad k \in \N,$$ where for the first inequality we used that $(n_k^{(j)})_{k \in \N} \subseteq (n_k^{(j-1)})$ for each $j \in \N,$ and for the second inequality we used that $(n_k^{(j)})_{k \in \N}$ is a subsequence of $\N$ for each $j \in \N.$ Now define the function $h:\R^d \to \R^d$ by
$$h(x) \coloneqq \lim_{k \to \infty} h_{m_k}(x).$$ 
Notice $h$ is well-defined since for each $j \in \N$, $(m_k)_{k \geq j}$ is a subsequence of $(n_k^{(j)})_{k \in \N}$ since each sequence is a subsequence of the previous, and $(h_{n_k^{(j)}}|_{\overline{B}(0,j)})_{k \in \N}$ is uniformly convergent, so $(h_{m_k}|_{\overline{B}(0,j)})_{k \in \N}$ is uniformly convergent. Hence, given a bounded set $\mathcal{K} \subset \R^d$, there exists $j \in \N$ such that $\mathcal{K} \subseteq \overline{B}(0,j),$ and so $(h_{m_k}|_\mathcal{K})_{k \in \N}$ is uniformly convergent to its pointwise limit $h|_\mathcal{K}.$

Finally, notice that for each $x,y \in \R^d$
$$||h(x)-h(y)|| = \bigg|\bigg|\lim_{k \to \infty} h_{m_k}(x) - \lim_{k \to \infty}h_{m_k}(y)\bigg|\bigg| = \lim_{k \to \infty}||h_{m_k}(x)-h_{m_k}(y)|| \leq  L||x-y||,$$ so $h$ is $L-$Lipschitz. 
\end{proof}

    \begin{lemma} \label{lemma: Y encodes rho} Let $\rho:\R^d \to \R$ with $0< \inf \rho \leq \sup \rho < \infty.$ Then there exists a Delone set $Y \subset \R^d$ such that $Y$ encodes $\rho,$ as in Definition \ref{defn: Y encoding rho}. Moreover, there exists $U_1 = U_1(\inf\rho) \in \N$ such that each closed cube with sidelength $U_1$ and vertices in $U_1(\Z-\frac{1}{2})^d$ contains at most $U_1^d\lceil \sup \rho \rceil$ points of $Y.$
    \end{lemma}
    \begin{proof}
        Define $(U_n)_{n \in \N}$ with $U_1 = 2\lceil (\frac{2}{\inf \rho})^\frac{1}{d}\rceil+1$ and 
    \begin{equation*}
        U_{n+1} = (2n-1)^2U_n, \qquad n \in \N.
    \end{equation*} Notice $U_n \xrightarrow{n \to \infty} \infty$, and since $U_1$ is odd, it follows $U_n$ is odd for every $n \in \N$. 
    Then define the sequence $(T_n)_{n \in \N}$ by $$T_n = (2n-1)U_n,\qquad n \in \N.$$  
    Let $C_0 \coloneqq \varnothing$ and $$C_n \coloneqq \left[-\tfrac{1}{2}U_{n},\tfrac{1}{2}U_{n}\right]^d, \qquad n \in \N.$$
    Now set $\mathcal{P}_0\coloneqq \{C_1\}$ and for each $n \in \N,$ denote by $\mathcal{P}_n$ the standard partition of $C_{n+1}$ into closed cubes of sidelength $U_n$ with vertices in ${\{-\frac{1}{2}}U_{n+1}\}^d+U_n\Z^d$.  Notice for each $n \in \N$ that $C_n \in \mathcal{P}_n$. Indeed, the sidelength of $C_n$ is $U_n,$ and since there are an odd number of cubes in $\mathcal{P}_n$, namely $(\frac{U_{n+1}}{U_n})^d = ((2n-1)^2)^d,$ there is a central cube and it must coincide with $C_n$ since both $C_n$ and $C_{n+1}$ have centre $\textbf{0}$. Thus, we see that $\bigcup_{n \in \N \cup \{0\}} (\mathcal{P}_n\setminus \{C_n\})$ is a partition of $\R^d$ whose sets have pairwise disjoint interiors.
    
    Now we construct the set $Y \subset \R^d$ as follows. For each $n \in \N \cup \{0\}$ and each cube $R \in \mathcal{P}_n \setminus \{C_n\}$, let $\{R_i\}_{i \in [{U_n}^d{U_1}^{-d}]}$ be the standard partition of $R$ into ${U_n}^d{U_1}^{-d}$ closed cubes with sidelength $U_1$ and vertices in $U_1(\Z-\frac{1}{2})^d,$ enumerated with an arbitrary ordering. Let $l \coloneqq {\lceil \sup \rho}\rceil^{-\frac{1}{d}}$, and for each $i \in [{U_n}^d{U_1}^{-d}]$, let $\{R_{i,j}\}_{j\in [U_1^dl^{-d}]}$ be the standard partition of $R_i$ into $U_1^dl^{-d}$ closed cubes of sidelength $l$, enumerated with an arbitrary ordering. By virtue of choosing $U_1 \geq (\frac{2}{\inf \rho})^\frac{1}{d}$ and noting $U_n \geq U_1$, we have that
    \begin{equation} \label{eq: Un U1}
        \bigg|\{R_i\}_{i \in [{U_n}^d{U_1}^{-d}]}\bigg|  =\tfrac{U_n^d}{U_1^d} \leq U_n^d \tfrac{\inf \rho}{2} \leq \lfloor U_n^d \inf \rho\rfloor.
    \end{equation} 
    Since $l^{-d} = \lceil \sup \rho\rceil$,
we can choose a collection of closed cubes $\mathcal{Q} \subseteq \{R_{i,j}\}_{i\in [{U_n}^d{U_1}^{-d}],j \in [l^{-d}]}$ with cardinality 
\begin{equation} \label{eq: cardinality of Q}
    |\mathcal{Q}| = \left\lfloor T_n^d\int_{\tfrac{1}{T_n}R}\rho \; d\mathcal{L}\right\rfloor \in [\lfloor U_n^d \inf \rho\rfloor, U_n^d \lceil \sup \rho \rceil].
\end{equation}
By \eqref{eq: Un U1} and \eqref{eq: cardinality of Q}, we have that $|\{R_i\}_{i \in [{U_n}^d{U_1}^{-d}]}| \leq \lfloor {U_n}^d \inf \rho\rfloor \leq|\mathcal{Q}|$, so we can further impose that for any $i \in [{U_n}^d{U_1}^{-d}]$ there exists $Q \in \mathcal{Q}$ such that $Q \subseteq R_i.$
    Then into each cube in $\mathcal{Q}$, place one point at its centre. This completes the construction of the set $Y.$
    
    Due to the two conditions for choosing $\mathcal{Q},$ notice for each $n \in \N \cup \{0\}$ and each cube $R \in \mathcal{P}_n \setminus \{C_n\}$ that 
    \begin{equation} \label{eq: up and down Ri Y}
        1 \leq |R_i \cap Y| \leq \bigg|\{R_{i,j}\}_{j \in [U_1^dl^{-d}]}\bigg| = U_1^dl^{-d} =U_1^d\lceil \sup \rho \rceil, \qquad i \in [{U_n}^d{U_1}^{-d}].
    \end{equation} 
   Since $\bigcup_{n \in \N \cup \{0\}} (\mathcal{P}_n\setminus \{C_n\})$ is a partition of $\R^d,$ then $\bigcup_{n \in \N \cup \{0\}}\bigcup_{R \in \mathcal{P}_n\setminus \{C_n\}}\{R_i\}_{i \in [{U_n}^d{U_1}^{-d}]}$ is also a partition of $\R^d,$ which by its definition consists closed cubes with sidelength $U_1$ and vertices in $U_1(\Z-\frac{1}{2})^d.$ By the second inequality of \eqref{eq: up and down Ri Y}, each of these cubes contains at most $U_1^d\lceil \sup\rho\rceil$ points of $Y$, as required.
   The first inequality of \eqref{eq: up and down Ri Y} further gives that each of these cubes contains at least one point of $Y$, so $Y$ has covering radius at most $U_1\sqrt{d}.$
   Moreover, since the sets $R_i$ partition $\R^d$ and each such set is further partitioned into smaller cubes $\{R_{i,j}\}_{j \in U_1^dl^{-d}},$ we have that $\bigcup_{n \in \N \cup \{0\}} \bigcup_{R \in \mathcal{P}_n \setminus \{\mathcal{C}_n\}}\bigcup_{i \in [U_n^dU_1^{-d}]}\{R_{i,j}\}_{j \in [U_1^dl^{-d}]}$ is a partition of $\R^d.$ Each cube $R_{i,j}$ has sidelength $l$ and either contains no point of $Y$, or it contains exactly one point of $Y$, positioned at its centre. Thus, $Y$ is $l-$separated, i.e., $Y$ has packing radius at least $\frac{l}{2}.$ Therefore, $Y$ is a Delone set.

   We now show $Y$ encodes $\rho.$
To this end, we will show the sequence of measures $(\mu_n)_{n \in \N}$ given by $$\mu_n(A) \coloneqq \tfrac{1}{T_n^d}|\tfrac{1}{T_n}Y \cap A|, \qquad A \subseteq \R^d, \qquad n \in \N$$  converges vaguely to $\rho \leb$ by showing $(\mu_n|_{[-\frac{N}{2},\frac{N}{2}]^d})_{n \in \N}$ converges weakly to $\rho\leb|_{[-\frac{N}{2},\frac{N}{2}]^d}$ for each odd natural number $N$. Fix an odd $N \in \N$ and invoke \cite[Proposition $4.1$]{bhat2025fast} with $\mathcal{K} = [-\frac{N}{2},\frac{N}{2}]^d,$  $\nu = \rho\leb|_{\mathcal{K}},$ and for each $n \geq \frac{N+1}{2},$ $\nu_n = \mu_n|_{[-\frac{N}{2},\frac{N}{2}]^d},$ and take $\mathcal{Q}_n$ to be the collection of $N^d(2n-1)^d$ closed cubes in the standard partition of $\mathcal{K}$ with sidelength $\frac{1}{2n-1}$ and vertices in $\{-\frac{N}{2}\}^d + \frac{1}{2n-1}\Z^d.$ 
Notice $\tfrac{1}{T_n}C_n \coloneqq  [-\frac{1}{2}\frac{U_n}{T_n}, \frac{1}{2}\frac{U_n}{T_n}]^d$ is in $\mathcal{Q}_n.$ Indeed, $\tfrac{1}{T_n}C_n$ has sidelength $\frac{U_{n}}{T_n} = \frac{1}{2n-1},$ and its centre is $\textbf{0}$ which coincides with the centre of $\mathcal{K}.$ Since $\mathcal{Q}_n$ contains an odd number of cubes, $N^d(2n-1)^d,$ it has a central cube which must coincide with $\tfrac{1}{T_n}C_n.$
As in the statement of \cite[Proposition $4.1$]{bhat2025fast}, take $\widetilde{\mathcal{Q}_n} \coloneqq \mathcal{Q}_n \setminus \{\tfrac{1}{T_n}C_n\}.$ The first and third conditions of \cite[Proposition $4.1$]{bhat2025fast} are clear. For the second condition, the Lebesgue measure of $\tfrac{1}{T_n}C_n$ is
\begin{equation*}
    \leb\left(\tfrac{1}{T_n}C_n\right) = \left(\tfrac{U_{n}}{T_n}\right)^d =  \left(\tfrac{1}{2n-1}\right)^d \xrightarrow{n \to \infty} 0 .
\end{equation*}Therefore, $\rho\leb(\tfrac{1}{T_n}C_n) \leq \sup \rho\cdot \leb(\tfrac{1}{T_n}C_n) \xrightarrow{n \to \infty} 0.$ Moreover, since for each $n \in \N$ we have that $C_n$ can be partitioned into $U_n^dU_1^{-d}$ cubes of sidelength $U_1$ with vertices in $U_1(\Z-\frac{1}{2})^d,$ and each such cube contains at most $U_1^d\lceil\sup \rho\rceil$ points of $Y$
\begin{equation} \label{eq: mun bounded}
    \nu_n\left(\tfrac{1}{T_n}C_n\right) = \tfrac{1}{T_n^d}\bigg|\tfrac{1}{T_n}Y \cap \tfrac{1}{T_n}C_n\bigg|= \tfrac{1}{T_n^d}|Y \cap C_n| \leq \tfrac{1}{T_n^d} U_1^{-d}{U_n^d}U_1^d\lceil\sup \rho\rceil \xrightarrow{n \to \infty} 0.
\end{equation}
It remains to verify the final condition of \cite[Proposition $4.1$]{bhat2025fast}. Henceforth, fix $n \geq \frac{N+1}{2}$ and $Q \in \widetilde{\mathcal{Q}_n}$. Notice that $T_nQ \in \mathcal{P}_n.$ Indeed, the sidelength of $T_nQ$ is $\frac{T_n}{2n-1} = U_n,$ and since $n \geq \frac{N+1}{2}$
$$T_nQ \subset \left[-\tfrac{NT_n}{2},\tfrac{NT_n}{2}\right]^d \subseteq\left[-\tfrac{(2n-1)T_n}{2},\tfrac{(2n-1)T_n}{2}\right]^d =\left[-\tfrac{1}{2}U_{n+1},\tfrac{1}{2}U_{n+1}\right]^d.$$
Moreover, the vertices of $T_nQ$ belong to $\{-\frac{1}{2}U_{n+1}\}^d+U_n\Z^d$. Indeed, $Q$ has vertices in $\{-\tfrac{N}{2}\}^d+\tfrac{1}{2n-1}\Z^d$, and so $T_nQ$ has vertices in 
$$T_n\left(\{-\tfrac{N}{2}\}^d+\tfrac{1}{2n-1}\Z^d\right) = \{-\tfrac{NT_n}{2}\}^d+{U_n}\Z^d.$$ Since $N$ is odd 
$$-\tfrac{NT_n}{2} + \tfrac{1}{2}U_{n+1}=-\tfrac{NT_n}{2} + \tfrac{(2n-1)T_n}{2} = T_n\tfrac{2n-1-N}{2}=U_n(2n-1)\tfrac{2n-1-N}{2} \in U_n\Z,$$ and hence $\{-\frac{NT_n}{2}\}^d+{U_n}\Z^d = \{-\frac{1}{2}U_{n+1}\}^d+U_n\Z^d,$ as required.
Thus, $T_nQ \in \mathcal{P}_n,$ and so we have that
\begin{equation*}
    \begin{split}
        \nu_n(Q) &= \tfrac{1}{T_n^d}\bigg|\tfrac{1}{T_n}Y \cap Q\bigg| =\tfrac{1}{T_n^d}\bigg|Y \cap T_nQ\bigg|
        = \tfrac{1}{T_n^d}\bigg\lfloor T_n^d\int_{\tfrac{1}{T_n}T_nQ}\rho\; d\leb\bigg\rfloor \leq \tfrac{1}{T_n^d}\bigg(T_n^d\int_{Q}\rho\; d\leb\bigg) = \rho\leb(Q),
    \end{split}
\end{equation*} and 
\begin{equation*}
    \nu_n(Q) = \tfrac{1}{T_n^d}\bigg\lfloor T_n^d\int_{Q}\rho\; d\leb\bigg\rfloor \geq \tfrac{1}{T_n^d}\left(T_n^d\int_{Q}\rho \; d\leb-1\right)=\rho\leb(Q) - \tfrac{1}{T_n^d}.
\end{equation*} Therefore,
\begin{equation*}
    |\nu_n(Q) - \rho\leb(Q)| \leq \tfrac{1}{T_n^d} \in o\left(\tfrac{1}{N^d(2n-1)^d}\right) = o\left(\tfrac{1}{|\mathcal{Q}_n|}\right).
\end{equation*} 
We have now verified all conditions of \cite[Proposition $4.1$]{bhat2025fast} for $(\nu_{n})_{n\in\N}$, $\nu$, $(\mathcal{Q}_{n})_{n\in\N}$ and $(\widetilde{\mathcal{Q}_{n}})_{n\in\N}$. Thus, for each odd $N \in \N$, $\mu_n|_{[-\frac{N}{2},\frac{N}{2}]^d}$ converges weakly to $\rho\leb|_{[-\frac{N}{2}, \frac{N}{2}]^d},$ and so by Proposition \ref{prop: vague weak prop}, $\mu_n$ converges vaguely to $\rho \leb.$ 
    \end{proof}
    
\maintheorem* 
\begin{proof} Fix a measurable density $\rho:\R^d \to \R$ with $0 < \inf \rho \leq \sup \rho < \infty.$ By Lemma \ref{lemma: Y encodes rho}, there exists a Delone set $Y \subset \R^d$ which encodes $\rho.$ In other words, there exists a sequence of positive numbers $(T_n)_{n \in \N} \to \infty$ such that the sequence of measures $(\mu_n)_{n \in \N}$ given by 
$$\mu_n(A) \coloneqq \tfrac{1}{T_n^d} \bigg|\tfrac{1}{T_n}Y \cap A\bigg|, \qquad  A \subseteq \R^d, \qquad \qquad n \in \N$$
converges vaguely to $\rho \leb.$
By assumption, there exists a Lipschitz, pointwise co-Lipschitz bijection $g:D \to E,$ where $D = Y$ and $E = \Z^d,$ or $D = \Z^d$ and $E = Y.$ While these two cases have meaningful differences, there are many similarities for the remainder of the proof, so we will manage both concurrently.

Let $L \coloneqq \text{Lip}(g).$ By Kirszbraun's extension theorem \cite{kirszbraun1934zusammenziehende}, we can extend $g$ to an $L-$Lipschitz map $\bar{g}:\R^d \to \R^d.$  Define the sequence of maps ${f_n:\R^d \to \R^d}$ as
\begin{equation*}
    f_n(x) \coloneqq \tfrac{1}{T_n}\bar{g}(T_nx).
\end{equation*} For each $n \in \N,$ $f_n$ is $L-$Lipschitz since for each pair $x,y \in \R^d$
\begin{equation*}
    ||f_n(x)-f_n(y)|| = \tfrac{1}{T_n}||\bar{g}(T_nx)-\bar{g}(T_ny)|| \leq \tfrac{1}{T_n}L||T_nx-T_ny|| = L||x-y||.
\end{equation*}We claim that $||f_n(0)|| \xrightarrow{n \to \infty} 0$. Fix $a \in D.$
    For each $n \in \N$, notice that $f_n(x) = \tfrac{1}{T_n}g(T_nx)$ on $\tfrac{1}{T_n}D.$ Let $S>0$ be the covering radius of $D$. Since there exists a point $x_n \in \tfrac{1}{T_n}D$ such that $||x_n-0|| = ||x_n|| \leq \tfrac{1}{T_n}S,$ we have 
    \begin{equation*}
        \begin{split}
            ||f_n(0)|| &\leq \bigg|\bigg|f_n(x_n)-\tfrac{1}{T_n}g(a)\bigg|\bigg| +\tfrac{1}{T_n}||g(a)|| + ||f_n(x_n)-f_n(0)||  \leq \tfrac{1}{T_n}(L||T_nx_n-a|| + ||g(a)|| + T_nL||x_n||)
            \\
            & \leq \tfrac{1}{T_n}(L(2S+||a||) + ||g(a)||) \xrightarrow{n \to \infty} 0.
        \end{split}
    \end{equation*} Hence, by Lemma \ref{lemma: lip extension on Rd}, there exists a $L-$Lipschitz map $f:\R^d \to \R^d$ and a strictly increasing sequence $(m_k)_{k \in \N} \subseteq \N$ such that for each bounded set $\mathcal{K} \subset \R^d$, $(f_{m_k}|_{\mathcal{K}})_{k \in \N}$ converges uniformly to $f|_{\mathcal{K}}$. Passing to a subsequence and relabelling, we assume henceforth that $(m_k)_{k \in \N} = \N.$

Define the sequence of counting measures $(\lambda_n)_{n \in \N}$ by
$$\lambda_n(A) \coloneqq \tfrac{1}{T_n^d} \bigg|\tfrac{1}{T_n}\Z^d \cap A\bigg|, \qquad A \subseteq \R^d.$$ It is clear that $(\lambda_n)_{n \in \N}$ converges weakly to $\leb.$
    \begin{claim} Taking $D = Y$ and $E = \Z^d,$ we have that
         $(f_n\#\mu_n)_{n \in \N}$ converges vaguely to $f \# (\rho \leb).$ Taking $D = \Z^d$ and $E = Y,$ we have that $(f_n\#\lambda_n)_{n \in \N}$ converges vaguely to $f \# \leb.$
    \end{claim} 
\begin{proof} Apply Lemma \ref{lemma: pushforward measures vaguely converging} with $h_n = f_n$ and $h = f,$ and when $D = Y$ and $E = \Z^d,$ we put $\nu_n = \mu_n$ and $\nu = \rho\leb,$ while when $D = \Z^d$ and $E = Y,$ we put $\nu_n = \lambda_n$ and $\nu = \leb.$ It remains to check conditions \eqref{eq: lemma condition 1} and \eqref{eq: lemma condition 2} from Lemma \ref{lemma: pushforward measures vaguely converging}. For condition \eqref{eq: lemma condition 1}, let $\mathcal{K}\subset \R^d$ be a bounded set. Then there exists a constant $C=C(\mathcal{K}) >0$ such that $||x||_\infty <C$ for each $x \in \mathcal{K}.$ 
Define the sequence of sets $(\mathcal{K}_n)_{n \in \N} \subseteq \R^d$ by
\begin{equation*}
    \mathcal{K}_n \coloneqq [-T_nC,T_nC]^d, \qquad n \in \N.
\end{equation*} 
Notice for each $n \in \N$ that $\mathcal{K} \subset [-C,C]^d = \frac{\mathcal{K}_n}{T_n}.$
Now by Lemma \ref{lemma: Y encodes rho}, there exists $U_1 = U_1(\inf \rho) \in \N$ such that each closed cube of sidelength $U_1$ and vertices in $U_1(\Z-\frac{1}{2})^d$ contains at most $\lceil \sup \rho \rceil$ points of $Y.$ 
For each $n \in \N,$ since the sidelength of $\mathcal{K}_n$ is $2T_nC$, notice we can cover $\mathcal{K}_n$ with at most $(\frac{2T_nC}{U_1}+2)^d$ closed cubes with sidelength $U_1$ and vertices in $U_1(\Z-\frac{1}{2})^d,$ 
and we can also cover $\mathcal{K}_n$ with at most $(2T_nC +2)^d$ closed cubes with sidelength $1$ and vertices in $(\Z-\frac{1}{2})^d.$ Each cube in the second covering contains exactly one point of $\Z^d.$ Moreover, since $(T_n)_{n \in \N}$ is a sequence of positive numbers going to $\infty,$ it is clear that $\inf T_n >0.$
Hence, for each $n \in \N$ 
$$\mu_n(\mathcal{K}) \leq \mu_n\left(\tfrac{\mathcal{K}_n}{T_n}\right)= \tfrac{1}{T_n^d}|\tfrac{Y}{T_n} \cap \tfrac{\mathcal{K}_n}{T_n}| =\tfrac{1}{T_n^d}|Y \cap \mathcal{K}_n| \leq \tfrac{1}{T_n^d}\left(\left(\tfrac{2T_nC}{U_1}+2\right)^d\lceil\sup \rho\rceil \right) \leq \left(\tfrac{2C}{U_1}+\tfrac{2}{\inf T_n}\right)^d  \lceil \sup\rho\rceil ,$$ and
$$\lambda_n(\mathcal{K}) \leq \lambda_n\left(\tfrac{\mathcal{K}_n}{T_n}\right)= \tfrac{1}{T_n^d}\bigg|\tfrac{\Z^d}{T_n} \cap \tfrac{\mathcal{K}_n}{T_n}\bigg| =\tfrac{1}{T_n^d}|\Z^d \cap \mathcal{K}_n|\leq \tfrac{1}{T_n^d}\left(2T_nC+2\right)^d \leq \left(2C+\tfrac{2}{\inf T_n}\right)^d .$$
For condition \eqref{eq: lemma condition 2}, we proceed without distinguishing between the different cases of $D$ and $E$. Recall $g: D\to E$ is a pointwise co-Lipschitz bijection. Let $a =g^{-1}(0)$. Then since $g$ is co-Lipschitz at $0,$ there exists $\beta>0$ such that $$||g(y)|| \geq \beta||y-a|| \qquad\text{for each}\; y \in D.$$ 
Fix $r > 0.$ Let $x \in g^{-1}(B(0,r)).$ Then $||g(x)|| < r,$ so $\beta||x-a|| <r,$ and so $x \in B(a, \frac{r}{\beta}) \subseteq B(0, ||a||+\frac{r}{\beta}).$ Therefore,
\begin{equation} \label{eq: preimage g}
    g^{-1}(B(0,r)) \subseteq B\left(0, ||a|| + \tfrac{r}{\beta}\right), \qquad r>0.
\end{equation}
Next fix $r>0$ and consider $x \in \overline{g}^{-1}(B(0,r)).$ By the relative density of $D,$ there exists a point $\widetilde{x} \in D$ such that $||x-\widetilde{x}|| \leq S.$ Then since $\overline{g}(\widetilde{x}) = g(\widetilde{x})$
$$||g(\widetilde{x})|| \leq ||\overline{g}(x) - \overline{g}(\widetilde{x})|| + ||\overline{g}(x)|| < L||x-\widetilde{x}|| + r \leq LS+r,$$
so by \eqref{eq: preimage g}, $||\widetilde{x}|| < ||a|| + \frac{r+LS}{\beta}.$ Consequently
$$||x|| \leq ||x-\widetilde{x}|| + ||\widetilde{x}|| < S + ||a|| + \tfrac{r+LS}{\beta},$$ so we have shown 
\begin{equation} \label{eq: preimage overline g}
    \overline{g}^{-1}(B(0,r)) \subseteq B\left(0, ||a|| + S + \tfrac{r+LS}{\beta}\right), \qquad r>0.
\end{equation}
Lastly, fix $r>0$ and $n \in \N.$ Then by \eqref{eq: preimage overline g}
\begin{equation} \label{eq: preimage fn}
    \begin{split}
        f_n^{-1}(B(0,r)) &= \{x \in \R^d: ||f_n(x)|| < r\} = \{x \in \R^d:\tfrac{1}{T_n}||\overline{g}(T_nx)|| <r\} = \tfrac{1}{T_n}\{y \in R^d:||\overline{g}(y)|| <T_nr\} \\
        &\subseteq \tfrac{1}{T_n}B\left(0,||a|| + S + \tfrac{T_nr+LS}{\beta}\right)\subseteq B\left(0, ||a|| +S+\tfrac{r+LS}{\beta}\right).
    \end{split}
\end{equation}
To complete the verification of condition \eqref{eq: lemma condition 2}, we show that $f^{-1}(B(0,r))$ is similarly contained. Fix $r>0$ and $x \in f^{-1}(B(0,r)).$ Since $(f_n)_{n \in \N}$ converges pointwise to $f$, there exists some $N=N_x \in \N$ such that $||f_N(x)-f(x)|| <1$ . Hence
\begin{equation*}
    ||f_N(x)|| \leq ||f(x)||+||f_N(x)-f(x)|| < r+1,
\end{equation*} so $x \in f_N^{-1}(B(0,r+1)).$ By \eqref{eq: preimage fn}, we therefore have that $x \in B(0,||a||+S+\frac{r+1+LS}{\beta}$), and so we obtain~\eqref{eq: lemma condition 2} with $R \coloneqq ||a||+S+\frac{r+1+LS}{\beta}$.
\end{proof}
Finally, we show the following claim which completes the proof of the theorem.
\begin{claim} Taking $D = Y$ and $E = \Z^d,$ we have that
    $(f_n \# \mu_n)_{n \in\N}$ converges weakly to $\leb.$ Taking $D = \Z^d$ and $E = Y,$ we have that $(f_n \# \lambda_n)_{n \in\N}$ converges vaguely to $\rho\leb.$
\end{claim}
\begin{proof} Let $A \subseteq \R^d$. First consider $D =Y$ and $E = \Z^d$. Then 
    \begin{equation*}
    \begin{split}
       (f_n \# \mu_n)(A) = \mu_n(f_n^{-1}(A)) &= \tfrac{1}{T_n^d}|f_n^{-1}(A) \cap \tfrac{1}{T_n}Y|=\tfrac{1}{T_n^d}|({f_n}|_{\frac{1}{T_n}Y})^{-1}(A) \cap \tfrac{1}{T_n}Y| = \tfrac{1}{T_n^d}|A \cap \tfrac{1}{T_n}\Z^d|= \lambda_n(A) .
       \end{split}
    \end{equation*} Since $(\lambda_n)_{n \in \N}$ converges weakly to $\leb$, we have that $(f_n\#\mu_n)_{n \in \N}$ converges weakly to $\leb$. 

    Similarly, if $D = \Z^d$ and $E = Y,$ \begin{equation*}
    \begin{split}
       (f_n \# \lambda_n)(A) = \lambda_n(f_n^{-1}(A)) &= \tfrac{1}{T_n^d}|f_n^{-1}(A) \cap \tfrac{1}{T_n}\Z^d|= \tfrac{1}{T_n^d}|({f_n}|_{\frac{1}{T_n}\Z^d})^{-1}(A) \cap \tfrac{1}{T_n}\Z^d| = \tfrac{1}{T_n^d}|A \cap \tfrac{1}{T_n}Y|=\mu_n(A).
       \end{split}
    \end{equation*} 
    The set $Y$ is chosen to encode $\rho$ as in Definition \ref{defn: Y encoding rho}, i.e., $(\mu_n)_{n \in \N}$ converges vaguely to $\rho \leb,$ and so $(f_n \# \lambda_n)_{n \in \N}$ converges vaguely to $\rho \leb,$ as required.  
\end{proof}
By the uniqueness of vague limits, when $D = Y$ and $E = \Z^d,$ since $(f_n\#\mu_n)$ simultaneously converges vaguely to $f \#(\rho \leb)$ and $\leb$, we have that $f\#(\rho\leb) = \leb.$ Similarly, when $D = \Z^d$ and $E = Y,$ we have that $f \# \leb = \rho \leb.$ For each $A \subset \R^d,$ first taking $D = Y$ and $E = \Z^d$ $$\leb(A) = (f\#(\rho\leb))(A) = \int_{f^{-1}(A)}\rho \; d\leb \geq (\inf\rho) \leb(f^{-1}(A)),$$ and next taking $D = \Z^d$ and $E = Y$
$$f \# \leb(A) = \leb(f^{-1}(A)) = \rho\leb(A) = \int_{A} \rho \; d\leb \leq (\sup\rho) \leb(A).$$ Thus, we have that $\frac{\leb(f^{-1}(A))}{\leb(A)}$ is uniformly bounded, so $f$ is Lipschitz regular by \cite[Lemma~$12.6$]{david1997fractured} (or \cite[Lemma~$2.4$]{dymond2018mapping}).
\end{proof}

\section{Proof of Theorem \ref{thm: pushforward}} \label{chapter: 1.11}
In this section, we will show the existence of a `bad' density $\rho:\R^d \to \R$ as defined by point \ref{enumerate: bad rho}' of the modified pushforward method. We will use results from \cite{dymond2018mapping} regarding the property of porosity.
\begin{define}[\cite{article}, Definition $2.1$]
    Let $(X,||\cdot||)$ be a Banach space.
    \begin{enumerate}[(i)]
        \item A set $P \subseteq X$ is called \emph{porous} at a point $x \in X$ if there exist $\epsilon_0>0$ and $\alpha \in (0,1)$ such that for every $\epsilon \in (0, \epsilon_0)$ there exists $y \in X$ such that $$||y-x|| < \epsilon\qquad \text{and} \qquad B(y,\alpha\epsilon) \cap P = \varnothing.$$ 
        \item A set $P \subseteq X$ is called \emph{porous} if $P$ is porous at $x$ for every point $x \in P.$
        \item A set $E \subseteq X$ is called $\sigma-$porous if $E$ is the countable union of porous subsets of $X.$
    \end{enumerate}
\end{define}
For our use later, we note that the class of $\sigma-$porous subsets of a Banach space $X$ is strictly contained in the class of subsets of $X$ of the first category, in the sense of the Baire Category Theorem \cite{article}. As a result, the complement of a $\sigma-$porous set is dense in $X.$
\begin{lemma} \label{lemma: existence of bad rho}
    Let $d \in \N$, $\mathcal{K} \subset \R^d$ be a compact convex set and $P \subseteq C(\mathcal{K},\R)$ be a $\sigma-$porous set in $(C(\mathcal{K},\R),||\cdot||_{\infty})$. Then the set 
    \begin{equation*}
        E_P \coloneqq \{\rho \in C_b(\R^d,\R): \rho|_{\mathcal{K}} \in P\} 
    \end{equation*} is a $\sigma-$porous set in $(C_b(\R^d,\R),||\cdot||_{\infty}).$ 
\end{lemma}
\begin{proof} As an important remark, $(C(\mathcal{K},\R), ||\cdot||_\infty)$ and $(C_b(\R^d,\R), ||\cdot||_\infty)$ are well-known Banach spaces. Now write $P = \bigcup_{n = 1}^\infty P_n$ where each set $P_n$ for $n \in \N$ is a porous subset of $C(\mathcal{K},\R)$. Consequently, we can write $$E_P = \bigg\{\rho \in C_b(\R^d,\R): \rho|_{\mathcal{K}} \in \bigcup_{n =1}^\infty{P_n}\bigg\} = \bigcup_{n =1}^\infty\{\rho \in C_b(\R^d,\R): \rho|_{\mathcal{K}} \in {P_n}\} = \bigcup_{n = 1}^\infty E_{P_n}.$$
    We will show $E_{P_n}$ is porous for each $n \in \N$. Fix $n \in \N$ and $\rho \in E_{P_n}.$ This means $\rho|_{\mathcal{K}} \in P_n$, so there exist $\epsilon_0>0$ and $\alpha \in (0,1)$ such that for all $\epsilon\in (0,\epsilon_0)$ there exists $\Phi =\Phi(\epsilon)\in C(\mathcal{K},\R)$ with
    \begin{equation} \label{eq: Phi condition}
        ||\rho|_{\mathcal{K}} - \Phi||_\infty< \epsilon \qquad \text{and} \qquad B_{(C(\mathcal{K},\R),||\cdot||_\infty)}(\Phi, \alpha \epsilon) \cap P_n = \varnothing.
    \end{equation}
   Fix $\epsilon \in (0,\epsilon_0),$ and let $\Phi \coloneqq \Phi(\epsilon)$ satisfy \eqref{eq: Phi condition}. Now define the function $\Psi:\R^d\to \R$ with the piecewise decomposition
    \begin{equation*}
        \Psi(x) =\begin{cases}
            \Phi(x), \qquad x \in \mathcal{K},\\
            \rho(x) + \Phi(\text{proj}(x)) - \rho(\text{proj}(x)), \qquad x \in \R^d \setminus \mathcal{K},
        \end{cases}
    \end{equation*} where $\text{proj}(x)$ is the Euclidean projection of $x$ onto $\mathcal{K},$ which is well-defined since $\mathcal{K}$ is convex \cite{bauschke2017convex}. Note we have for each $x \in \R^d$ that $|\Psi(x) - \rho(x)| < \epsilon.$ Indeed, for each $x \in \mathcal{K}$ we have that $$|\Psi(x)-\rho(x)| = |\Phi(x) - \rho(x)| \leq ||\Phi-\rho|_{\mathcal{K}}||_\infty < \epsilon,$$ and for each $x \in \R^d \setminus\mathcal{K}$ $$|\Psi(x)-\rho(x)| = |\Phi(\text{proj}(x)) - \rho(\text{proj}(x))| \leq  ||\Phi-\rho|_{\mathcal{K}}||_\infty < \epsilon.$$ This also gives us that $\Psi$ is bounded. Additionally, when $x \in \partial\mathcal{K}$ we see that $$\rho(x) + \Phi(\text{proj}(x)) - \rho(\text{proj}(x)) = \rho(x) + \Phi(x) - \rho(x) = \Phi(x),$$ and since $\Phi$ and $\rho$ are continuous, and the projection $\text{proj}(x)$ is onto a closed convex set and is thus continuous \cite[Proposition $4.8$]{bauschke2017convex}, we have that $\Psi$ is continuous. In other words, we have shown $\Psi \in C_b(\R^d,\R)$ satisfies $$||\rho-\Psi||_\infty < \epsilon.$$
    Finally, we will show that $$B_{(C_b(\R^d,\R),||\cdot||_\infty)}(\Psi, \alpha \epsilon) \cap E_{P_n} = \varnothing.$$ Let $\zeta \in C_b(\R^d,\R)$ such that $||\Psi-\zeta||_\infty < \alpha\epsilon.$ Then $$||\Phi-\zeta|_{\mathcal{K}}||_\infty =||\Psi|_{\mathcal{K}}-\zeta|_{\mathcal{K}}||_\infty < \alpha\epsilon,$$ so $\zeta|_\mathcal{K} \in B_{(C(\mathcal{K},\R),||\cdot||_\infty)}(\Phi, \alpha \epsilon).$ As a result, $\zeta|_\mathcal{K} \notin P_n$ and so $\zeta \notin E_{P_n},$ as required. Therefore, $E_{P_n}$ is porous, and $E_P = \bigcup_{n=1}^\infty E_{P_n}$ is $\sigma-$porous.
    \end{proof}
    
\begin{lemma} \label{lemma: porosity e}
    Let $d \in \N_{\geq 2}$ and $\mathcal{E} \coloneqq \{\rho \in C_b(\R^d,\R): \exists f:\R^d \to \R^d \;\text{Lipschitz regular with}\; f\#\rho\leb =\leb\}$. Then $\mathcal{E}$ is a $\sigma-$porous subset of $(C_b(\R^d,\R),||\cdot||_{\infty}).$
\end{lemma}
\begin{remark*}
    For $\rho\in C_{b}(\R^{d},\R)$ attaining negative values, we interpret $f\#(\rho\leb)$ as the signed measure given by $f\#(\rho\leb)=f\#(\rho^{+}\leb)-f\#( \rho^{-}\leb)$, as in \cite{dymond2018mapping}. Here, $\rho^{+}$ and $\rho^{-}$ stand for the positive and negative parts of $\rho$ respectively.
\end{remark*}
\begin{proof}[Proof of Lemma~\ref{lemma: porosity e}]
     We will use a modified version of the setup from \cite[Porous decomposition of $\mathcal{E}$]{dymond2018mapping}. Fix $d \in \N_{\geq 2}$ and $M \in \N$.  Let $(O_{M,n})_{n \in \N}$ be a countable topological basis of $[-M,M]^d.$ 
For $C,L,n \in \N$ we let $\mathcal{E}_{M,C,L,n}$ be the set of positive
continuous functions $\rho:[-M,M]^d \to \R$ for which the following holds: there are pairwise disjoint, non-empty open sets $Y_1,...,Y_N \subset [-M,M]^d$ with $Y_1 = O_{M,n}$, where $N \in [C]$, an open set $V \subset \R^d$, and a family of $(\frac{1}{2C^2}, L)-$bi-Lipschitz homeomorphisms $f_i
: Y_i \to V$ such that
\begin{equation} \label{eq: ecln jacobian criterion}
\begin{split}
    &\rho(y) = |\text{Jac}(f_1)(y)| - \sum_{i=2}^N (\rho\circ f_i^{-1}\circ f_1)(y)|\text{Jac}(f_i^{-1} \circ f_1)(y)| \qquad \text{for}\; a.e. \; y \in O_{M,n}.
    \end{split}
\end{equation}
By \cite[Lemma $4.7$]{dymond2018mapping}, wherein the replacement of $I^d \coloneqq [0,1]^d$ with $[-M,M]^d$ does not affect the proof, we have that 
$\mathcal{E}_{M,C,L,n}$ is a porous subset of ${(C([-M,M]^d, \R),||\cdot||_{\infty}).}$ Thus, for each $M \in \N$, the set $\bigcup_{C,L,n \in \N}\mathcal{E}_{M,C,L,n}$ is $\sigma-$porous. By Lemma \ref{lemma: existence of bad rho}, the set $$A_M \coloneqq \bigg\{\rho\in C_b(\R^d,\R):\rho|_{[-M,M]^d}\in  \bigcup_{C,L,n \in \N}\mathcal{E}_{M,C,L,n}\bigg\}$$ is $\sigma-$porous in $(C_b(\R^d,\R),||\cdot||_{\infty}).$ As a result, $A \coloneqq \bigcup_{M=1}^\infty A_M$ is $\sigma-$porous in $(C_b(\R^d,\R),||\cdot||_{\infty})$.

 We now show that $\mathcal{E} \subseteq A$. Fix $\rho \in \mathcal{E}.$ Then there exists a Lipschitz regular function $f:\R^d \to \R^d$ with $f\#(\rho\leb) = \leb$. 
 By \cite[Proposition~$2.15$]{dymond2018mapping}, there exist 
 $N \leq \text{Reg}(f),$ pairwise disjoint, non-empty subsets $\widetilde{Y}_1,...,\widetilde{Y}_{N} \subset \R^d$ and a non-empty open set $\widetilde{V} \subseteq \R^d$ such that $$\bigcup_{i=1}^{N} \widetilde{Y}_i = f^{-1}(\widetilde{V})$$ and
 the map $f$ can be `decomposed' into bi-Lipschitz bijections  $$f|_{\widetilde{Y}_i} \eqqcolon f_i:\widetilde{Y}_i \to \widetilde{V}\qquad \text{with} \;i \in [N].$$ Moreover, their result allows for each $i \in [N]$ that the map $f_i$ is $( \frac{1}{2\text{Reg}(f)^2},\text{Lip}(f))-$bi-Lipschitz. Now take an arbitrary open ball $V \subseteq \widetilde{V}$. Since $f_i$ is $( \frac{1}{2\text{Reg}(f)^2},\text{Lip}(f))-$bi-Lipschitz for each $i \in [N],$ the preimage $f^{-1}(V)$ is the union of a collection of pairwise disjoint bounded open sets $Y_1,...,Y_{N} \subset \R^d$ given by
 $$Y_i = f_i^{-1}(V),\qquad i \in [N].$$ 
In particular, there exists some $M \in \N$ such that
$$f^{-1}(V) = \bigcup_{i=1}^NY_i \subseteq [-M,M]^d,$$ so we have that
  \begin{equation} \label{eq: pushforward restriction}
     (f\#(\rho|_{[-M,M]^d}\leb))|_{V} = \leb|_{V}.
 \end{equation} 
Let $B\subseteq V$ be a measurable set. By \eqref{eq: pushforward restriction}
\begin{equation*}
\begin{split}
    \int_{B}d\leb =\leb(B) = (f\#(\rho|_{[-M,M]^d}\leb))(B) = \int_{f^{-1}(B)} \rho|_{[-M,M]^d}\; d\leb &= \sum_{i=1}^N \int_{f_i^{-1}(B)}\rho|_{[-M,M]^d} \; d\leb \\&=\int_{B}\sum_{i=1}^N(\rho|_{[-M,M]^d}\circ f_i^{-1}) |\text{Jac}(f_i^{-1})| \; d\leb. 
    \end{split}
\end{equation*}
Since $f_1$ is bi-Lipschitz, we can write $B = f_1(f_1^{-1}(B))$ and apply a change of variables to both sides to get
\begin{equation*}
\begin{split}
    \int_{f_1^{-1}(B)}|\text{Jac}(f_1)|\;d\leb &=  \int_{f_1^{-1}(B)}\left(\left(\sum_{i=1}^N(\rho|_{[-M,M]^d}\circ f_i^{-1}) |\text{Jac}(f_i^{-1})|\right)\circ f_1\right)|\text{Jac}(f_1)| \; d\leb \\&= \int_{f_1^{-1}(B)}\sum_{i=1}^N(\rho|_{[-M,M]^d}\circ f_i^{-1}\circ f_1) |\text{Jac}(f_i^{-1})\circ f_1||\text{Jac}(f_1)| \; d\leb \\&= \int_{f_1^{-1}(B)}\sum_{i=1}^N(\rho|_{[-M,M]^d}\circ f_i^{-1}\circ f_1) |\text{Jac}(f_i^{-1}\circ f_1)| \; d\leb,
    \end{split}
\end{equation*} where for the second equality we used right-distributivity of composition over addition and multiplication, and for the final equality we used the chain rule for Jacobians. Since $B\subseteq V$ is an arbitrary measurable set, the integrands must agree almost everywhere on $f_1^{-1}(V) = Y_1$, i.e.
\begin{equation*}
      \rho|_{[-M,M]^d}(y) =|\text{Jac}(f_1)(y)|-\sum_{i=2}^N (\rho|_{[-M,M]^d}\circ f_i^{-1}\circ f_1)(y)|\text{Jac}(f_i^{-1} \circ f_1)(y)| \qquad \text{for} \;a.e. \;y \in Y_1.
 \end{equation*} 
 Now choose $n \in \N$ such that $O_{M,n} \subseteq Y_1.$
Then we have shown that $\rho|_{[-M,M]^d} \in \mathcal{E}_{M,C,L,n}$ for $C = \text{Reg}(f)$ and $L = \text{Lip}(f),$ meaning $\rho \in A_M \subseteq A,$ as required. Hence, by the $\sigma-$porosity of $A$ we have that $\mathcal{E}$ is $\sigma-$porous in $(C_b(\R^d,\R), ||\cdot||_{\infty}).$
\end{proof}

\begin{lemma} \label{lemma: porosity f}
     Let $d \in \N_{\geq 2}$ and $\mathcal{F} \coloneqq \{\rho \in C_b(\R^d,\R): \exists f:\R^d \to \R^d \;\text{Lipschitz regular with}\; f\#\leb =~\rho\leb\}$. Then $\mathcal{F}$ is a $\sigma-$porous subset of $(C_b(\R^d,\R),||\cdot||_{\infty}).$
\end{lemma}
\begin{proof}
     Fix $(O_n)_{n \in \N}$ to be a countable topological basis for $\R^d$ consisting only bounded sets. For each $C,L,n \in \N$ let 
    \begin{equation*}
 \mathcal{F}_{C,L,n} \coloneqq \left\{
  \rho\in C_b(\R^d,\R) \;\middle|\;
  \begin{aligned}
  & \exists N \in [C], \;h_1,...,h_N:O_n \to \R^d \text{ all }(\tfrac{1}{L},2C^2)\text{-bi-Lipschitz such that}\\
  &\qquad \qquad \rho(x) = \sum_{i=1}^N|\text{Jac}(h_i)(x)| \text{ for $a.e$ $x \in O_n$}.
  \end{aligned}
\right\}
\end{equation*}
First we show that $\mathcal{F} \subseteq \bigcup_{C,L,n \in \N}\mathcal{F}_{C,L,n}.$ Indeed, if $\rho \in \mathcal{F},$ then there exists a Lipschitz regular function $f:\R^d \to \R^d$ with $f\#\leb = \rho\leb.$ Let $L \coloneqq \text{Lip}(f)$ and $C \coloneqq \text{Reg}(f).$ By \cite[Proposition $2.15$]{dymond2018mapping}, there exist $N \in [C]$, an open set $V \subset \R^d,$ pairwise disjoint, non-empty open sets $Y_1,...,Y_N \in \R^d$ such that $\bigcup_{i \in [N]} Y_i = f^{-1}(V)$, and $(\frac{1}{2C^2},L)$-bi-Lipschitz bijections $f_i: Y_i \to V$ given by $f_i = f|_{Y_i}$ for each $i \in [N].$ Now let $A \subseteq V$ be a measurable set. Then
    \begin{equation*}
    \begin{split}
        \int_{A}\rho \; d \leb &= \rho \leb(A) = (f\#\leb)(A) = \int_{f^{-1}(A)}d \leb \\
        & = \sum_{i=1}^N \int_{f_{i}^{-1}(A)}d \leb  = \sum_{i=1}^N \int_{A}|\text{Jac}(f_i^{-1})| \; d \leb = \int_{A}\sum_{i=1}^N|\text{Jac}(f_i^{-1})| \; d \leb.
    \end{split}
\end{equation*}
Since $A \subseteq V$ is an arbitrary measurable set, the integrands must agree almost everywhere on $V$, so we have that $\rho(x) = \sum_{i=1}^N |\text{Jac}(h_i)(x)|$ for $a.e.$ $x \in V,$ where $h_i = f_i^{-1}$ for each $i \in [N].$ Now we choose $n \in \N$ satisfying $O_n \subseteq V$ to conclude that $\rho \in \mathcal{F}_{C,L,n}.$ 

Fix $C,L,n \in \N$, $\phi \in C_b(\R^d, \R)$ and $\epsilon \in (0,1).$ We now construct a function $\widetilde{\phi} \in C_b(\R^d,\R)$ such that $||\phi - \widetilde{\phi}||_{\infty} \leq \epsilon$ and $B_{||\cdot||_{\infty}}(\widetilde{\phi}, \frac{\epsilon}{3+C}) \cap \mathcal{F}_{C,L,n} = \varnothing$. First, since $\phi$ is uniformly continuous on $\overline{O_n},$ we can choose $\delta>0$ sufficiently small so that  
\begin{equation} \label{eq: uniform continuity}
    |\phi(x)-\phi(y)| \leq \tfrac{\epsilon}{3+C} \qquad \text{whenever $x,y \in O_n$ with $||x-y|| \leq \delta$}.
\end{equation}
Now choose an open set $U \subseteq O_n$ with $\diam(U) \leq \delta.$
By \cite[Lemma~4.6]{dymond2018mapping} there exists a function $\psi\in C_{b}(\R^{d},\R)$ such that $\infnorm{\psi}\leq \varepsilon$, $\operatorname{supp}(\psi)\subseteq U$ and for every $k$-tuple of $(\tfrac{1}{L},2C^2)$-bilipschitz mappings $h_{i}\colon U\to \R^{d}$ there exists $\mathbf{e}_1-$adjacent cubes $S,S' \subseteq U$ such that
    \begin{equation} \label{eq: jac jac}
        \tfrac{1}{\leb(S)}\bigg|\int_S |\text{Jac}(h_i)| \; d\leb - \int_{S'}|\text{Jac}(h_i)| \; d \leb \bigg| \leq \tfrac{\epsilon}{3+C}
    \end{equation} and
    \begin{equation} \label{eq: psi psi}
        \tfrac{1}{\leb(S)}\bigg|\int_S \psi\; d \leb - \int_{S'}\psi \; d \leb\bigg| \geq \epsilon.
    \end{equation} 
We define $\widetilde{\phi} \in C_b(\R^d, \R)$ to be given by
$$\widetilde{\phi} \coloneqq \phi + \psi.$$
Then $||\phi - \widetilde{\phi}||_{\infty} = ||\psi||_{\infty} \leq \epsilon,$ as required. Now let $\rho \in B_{||\cdot||_{\infty}}(\widetilde{\phi},\frac{\epsilon}{3+C}),$ and suppose for a contradiction that $\rho \in \mathcal{F}_{C,L,n}.$ This means there exist $N \in [C]$ and $(\frac{1}{L},2C^2)-$bi-Lipschitz maps $h_1,...,h_N:O_n \to \R^d$ such that 
$$\rho(x) = \sum_{i=1}^N |\text{Jac}(h_i)(x)|\qquad \text{for $a.e. \; x \in O_n$}.$$
Notice that
\begin{equation}\label{eq: psi decomposition}
    \begin{split}
        \psi(x) = \widetilde{\phi}(x)-\rho(x) + \rho(x) - \phi(x) =  \widetilde{\phi}(x)-\rho(x) + \sum_{i=1}^N |\text{Jac}(h_i)(x)| - \phi(x), \qquad a.e.\; x \in O_n.
    \end{split}
\end{equation}  Now let $S,S' \subseteq U$ be the cubes satisfying \eqref{eq: jac jac} and \eqref{eq: psi psi} and for the two sides of \eqref{eq: psi decomposition}, consider the size of the difference between their averages over $S$ and $S'$. By \eqref{eq: psi psi},
the left-hand difference is at least $\epsilon.$ On the other hand, since first we have that $||\widetilde{\phi}-\rho||_{\infty} < \frac{\epsilon}{3+C},$ then we have \eqref{eq: jac jac}, and finally we have $|\phi(x)-\phi(y)| \leq \frac{\epsilon}{3+C}$ for each $x,y \in S \cup S'$ (using $\diam(S \cup S') \leq \diam(U) \leq \delta$ and \eqref{eq: uniform continuity}), the right-hand difference is strictly less than $$2\tfrac{\epsilon}{3+C} + N\tfrac{\epsilon}{3+C} + \tfrac{\epsilon}{3+C} \leq 2\tfrac{\epsilon}{3+C} + C\tfrac{\epsilon}{3+C} + \tfrac{\epsilon}{3+C} = \epsilon,$$ which is a contradiction. Therefore, $\rho$ does not belong to $\mathcal{F}_{C,L,n}.$ Thus, $\mathcal{F}_{C,L,n}$ is porous, and so we conclude that $\mathcal{F}$ is $\sigma-$porous in $(C_b(\R^d,\R), ||\cdot||_{\infty}).$
\end{proof}
With this, the main theorem of this section follows as a corollary of Lemmas \ref{lemma: porosity e} and \ref{lemma: porosity f}.
\pushforwardthm*
\begin{proof}
 The set described in the statement of this theorem is precisely the union of $\mathcal{E}$ and $\mathcal{F}$ from Lemmas~\ref{lemma: porosity e} and \ref{lemma: porosity f} respectively, which are both $\sigma-$porous, and so is itself $\sigma-$porous.
\end{proof} 
\section{Lipschitz self-bijections of $\Z^d$ with different co-Lipschitz properties} \label{chapter 4}
   In this section we show that the two forward implications in the conclusion of Theorem~\ref{thm:summary} are, in general, strict. In other words, we show that Lipschitz bijections between two Delone sets in $\R^{d}$ may fail to be co-Lipschitz at every point and we verify that Lipschitz and pointwise co-Lipschitz bijections may fail to be bilipschitz. 
   
   Among Lipschitz bijections between Delone sets, this chapter emphasises that the class of pointwise co-Lipschitz maps is strictly broader than the class of bi-Lipschitz maps. This is shown in Lemma \ref{lemma: bilipschitz properties}, for which we construct a Lipschitz bijection $g^{(d)}:\Z^d \to \Z^d$ dependent on a sequence $(N_i)_{i \in \N} \subseteq \N,$ such that different candidates for $(N_i)_{i \in \N}$ give $g^{(d)}$ different co-Lipschitz properties. In particular, there exist sequences such that $g^{(d)}$ is pointwise co-Lipschitz but is not bi-Lipschitz.

\begin{lemma} \label{lemma: lip extension higher dim}
    Let $d \in \N$, and for each $i \in [d]$ let $L_i >0$ and the mapping $h_i:\R \to \R$ be $L_i-$Lipschitz. Then the mapping $g^{(d)}:\R^d \to \R^d$ given by $$g^{(d)}(x_1,...,x_d) \coloneqq (h_1(x_1),h_2(x_2),...,h_d(x_d))$$ is $(\max_{i \in [d]}L_i)-$Lipschitz.
\end{lemma}
\begin{proof}
    Let $X,Y \in \R^d$ be denoted as $X = (x_1,...,x_d)$ and $Y = (y_1,...,y_d).$ Then
    \begin{equation*}
        \begin{split}
            ||g^{(d)}(X)-g^{(d)}(Y)||^2 
            &= (h_1(x_1)-h_1(y_1))^2 + ...+ (h_d(x_d)-h_d(y_d))^2 \leq L_1^2(x_1-y_1)^2 + ... + L_d^2(x_d+y_d)^2 \\
            & \leq \left(\max_{i \in [d]}L_i\right)^2((x_1-y_1)^2+... + (x_d -y_d)^2) = \left(\max_{i \in [d]} L_i\right)^2 ||X-Y||^2.
        \end{split}
    \end{equation*}
\end{proof}
\begin{lemma} \label{lemma: bilip at 0} \label{lemma: not bilip at pt}
    Let $d \in \N$, $U \subseteq \R$ a Delone set with $0 \in U$, $g:U \to U$ a Lipschitz map and let $g^{(d)}:U^d \to U^d$ be given by
    $$g^{(d)}(x_1,...,x_d) \coloneqq (g(x_1),x_2,...,x_d).$$ Then $g$ is pointwise co-Lipschitz if and only if $g^{(d)}$ is pointwise co-Lipschitz.
\end{lemma}
\begin{proof} Throughout the proof, we will use the dichotomy from Proposition \ref{thm: bilip}, that $g$ or $g^{(d)}$, respectively, are either pointwise co-Lipschitz or nowhere co-Lipschitz. Thus, it is sufficient to check whether $g$ or $g^{(d)}$ are co-Lipschitz at $g(0)$ or $g^{(d)}(0,0,...,0)$ respectively.

    First, if $g$ is co-Lipschitz at $g(0),$ there exists $\beta>0$ such that for all $x \in U$
    $${|g(x)-g(0)|}\geq \beta|x|.$$
    Then for each non-zero $X = (x_1,...,x_d) \in U^d$
    \begin{equation*}
        \begin{split}
            \frac{||g^{(d)}(X)-g^{(d)}(0,0,...,0)||^2}{||X||^2} &= \frac{(g(x_1)-g(0))^2+x_2^2+...+x_d^2}{x_1^2+...+x_d^2} \geq \frac{\beta^2x_1^2+x_2^2+...+x_d^2}{x_1^2+...+x_d^2} = 1+\frac{(\beta^2-1)x_1^2}{x_1^2+...+x_d^2}\\&= 1+\frac{\beta^2-1}{1+(\frac{x_2}{x_1})^2+...+(\frac{x_d}{x_1})^2}\geq \min\{1,1+{\beta^2-1}\}=\min\{1,\beta^2\}>0.
        \end{split}
    \end{equation*}
    Conversely, if $g$ is not co-Lipschitz at $g(0),$ there exists a sequence $(x_i)_{i \in \N} \subseteq U\setminus \{0\}$ such that for each $i \in \N$
    $$\frac{|g(x_i)-g(0)|}{|x_i-0|}<\frac{1}{i}.$$
Now choose the sequence $(X_i)_{i \in \N} \subseteq U^d$ with
$X_i = (x_i,0,0,...,0).$
Then for each $i \in \N$
$$\frac{||g^{(d)}(X_i) - g^{(d)}(0,0,...,0)||}{||X_i - (0,0,...,0)||} = \frac{|g(x_i)-g(0)|}{|x_i-0|}<\frac{1}{i},$$
so $g^{(d)}$ is not co-Lipschitz at $g^{(d)}(0,0,...,0).$
\end{proof} 

For the remainder of the present section, we will describe a particular construction for a Lipschitz bijection of $\Z^d,$ denoted $g^{(d)}:\Z^d \to \Z^d.$ It will be built as follows. First, for a given $N \in \N$, we define the `stretch and $3-$fold' bijection $h^{(N)}:[3N] \to [3N].$ Then we construct a Lipschitz bijection $g_+:\N \to \N$ by breaking $\N$ into intervals (see Figure \ref{figure: bijection}) and applying a translated `stretch and $3-$fold' to each. We then lift $g_+$ to a Lipschitz bijection $g:\Z \to \Z$ in the usual way, given by \eqref{eq: definition of f}. In choosing the lengths of these intervals, we will be able to control whether $g$ is pointwise co-Lipschitz or not, and whether $g$ is bi-Lipschitz or not. Finally, we lift $g$ to a Lipschitz bijection $g^{(d)}:\Z^d \to \Z^d$, given by \eqref{eq: definition of fd}, retaining the bi-Lipschitz properties of $g.$

The motivation of the following construction is analogised by stretching and then folding a line segment into a `$Z-$shape' (see Figure \ref{figure: Z}). Given $N \in \N,$ we define the bijection $h^{(N)}:[3N] \to [3N]$ as
\begin{equation} \label{eq: hN}
    h^{(N)}(x) = \begin{cases}
        3x, \qquad\qquad\qquad\;\; 0<x\leq N\\
        6N+2-3x,\qquad N<x\leq 2N\\
        3x-6N-2,\qquad 2N<x\leq 3N.
    \end{cases}
\end{equation} 

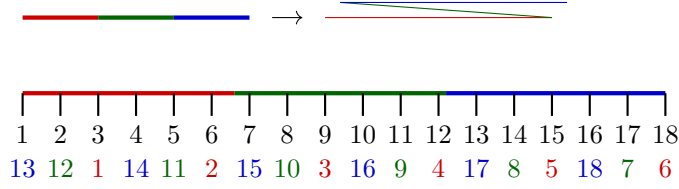
\begin{figure}[h]
    \centering
    \begin{tikzpicture}
        \draw[red,ultra thick] (0,1) -- (1,1);
        \draw[green,ultra thick] (1,1) -- (2,1);
        \draw[blue,ultra thick] (2,1) -- (3,1);
        \draw[red] (4,1)--(7,1);
        \draw[green] (7,1) -- (4.2,1.2);
        \draw[blue] (4.2,1.2) -- (7.2,1.2);
        \draw[->] (3.3,1) -- (3.7,1);
        \draw [ultra thick,red] (0,0) -- (2.8,0);
        \draw [ultra thick,green] (2.8,0) -- (5.6,0);
        \draw [ultra thick,blue] (5.6,0) -- (8.5,0);
    \foreach \i [count=\x from 0] in  {1,2,...,18}
\draw[thick] (\x/2,0mm) -- ++ (0,-3mm) node (n\x) [below] {$\i$};
\draw[blue] node at (0,-1) {$13$};
\draw[green] node at (0.5,-1) {$12$};
\draw[red] node at (1,-1) {$1$};
\draw[blue] node at (1.5,-1) {$14$};
\draw[green] node at (2,-1) {$11$};
\draw[red] node at (2.5,-1) {$2$};
\draw[blue] node at (3,-1) {$15$};
\draw[green] node at (3.5,-1) {$10$};
\draw[red] node at (4,-1) {$3$};
\draw[blue] node at (4.5,-1) {$16$};
\draw[green] node at (5,-1) {$9$};
\draw[red] node at (5.5,-1) {$4$};
\draw[blue] node at (6,-1) {$17$};
\draw[green] node at (6.5,-1) {$8$};
\draw[red] node at (7,-1) {$5$};
\draw[blue] node at (7.5,-1) {$18$};
\draw[green] node at (8,-1) {$7$};
\draw[red] node at (8.5,-1) {$6$};
    \end{tikzpicture}
    \caption{The `stretch and $3-$fold' mapping $h^{(6)}:[18] \to [18].$}
    \label{figure: Z}
\end{figure}
Let $(N_i)_{i \in \N} \subseteq \N$. Define the sequence of `base points' $(s_i)_{i \in \N}$ with $s_1 = 0$ and
\begin{equation} \label{eq: si}
    s_i \coloneqq \sum_{j=1}^{i-1}3N_j,\qquad i \geq 2.
\end{equation}
Then define the bijection $g_+:\N \to \N$ as
\begin{equation*} 
    g_+(x) = 
    s_{i}+h^{(N_i)}(x-s_{i}) \qquad \;\;\text{whenever} \; x \in s_{i}+[3N_i] \; \text{for} \; i \in \N.
\end{equation*}
This is indeed a bijection since each $x \in \N$ lies in a unique interval either of the form $s_{i}+[3N_i]$ (see Figure~\ref{figure: bijection}), and for each $i \in \N$ the `stretch and $3-$fold' mapping $h^{(N_i)}$ on $[3N_i]$ is a self-bijection. Extend $g_+$ to a bijection $g: \Z \to \Z$ by prescribing that 
\begin{equation}\label{eq: definition of f}
    g(0) = 0, \qquad g|_{\N} = g_+ \qquad \text{and} \qquad g(x) = -g_+(-x) \qquad \text{for all $x \in -\N$}.
\end{equation}

\usetikzlibrary{decorations.pathreplacing}
\begin{figure}[h]
    \centering
    \begin{tikzpicture}
        \draw[black, ultra thick] (0,0) -- (15,0);
\draw[thick] (0,0mm) -- ++ (0,-4mm) node [below] {$0=s_1$};
\draw[thick] (1,0mm) -- ++ (0,-1mm) node [below] {$1$};
\draw[thick] (3,0mm) -- ++ (0,-4mm) node [below] {$3N_1 = s_2$};
\draw[thick] (4,0mm) -- ++ (0,-1mm) node [below] {$s_2+1$};
\draw[thick] (7,0mm) -- ++ (0,-4mm) node [below] {$s_2+3N_2 = s_3$};
\draw[thick] (8,0mm) -- ++ (0,-1mm) node [below] {$s_3+1$};
\draw[thick] (11,0mm) -- ++ (0,-4mm) node [below] {$s_3+3N_3 = s_4$};
\draw[thick] (12,0mm) -- ++ (0,-1mm) node [below] {$s_4+1$};
\node at (13,-1) {$\cdots$};
\draw [decorate,decoration={brace,amplitude=5pt,raise=4ex}]
  (1,0) -- (3,0) node[midway,yshift=3em]{$3N_1$};
  \draw [decorate,decoration={brace,amplitude=5pt,raise=4ex}]
  (4,0) -- (7,0) node[midway,yshift=3em]{$3N_2$};
  \draw [decorate,decoration={brace,amplitude=5pt,raise=4ex}]
  (8,0) -- (11,0) node[midway,yshift=3em]{$3N_3$};
  \draw [decorate,decoration={brace,amplitude=5pt,raise=4ex}]
  (12,0) -- (15,0) node[midway,yshift=3em]{$3N_4$};
    \end{tikzpicture}
    \caption{The intervals on which $g_+$ is defined. For each $i \in \N$, on each interval of $3N_i$ natural numbers, the translated `stretch and $3-$fold' bijection is applied.}
    \label{figure: bijection}
\end{figure}
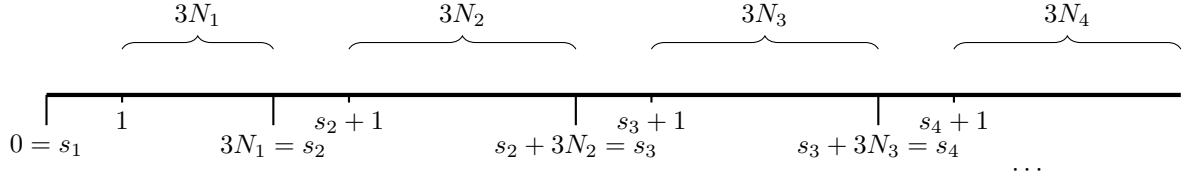
\begin{lemma} \label{lemma: 5-lipschitz}
    For any sequence $(N_i)_{i \in \N} \subseteq \N$, the bijection $g:\Z \to \Z$ given by \eqref{eq: definition of f} is $5-$Lipschitz.
\end{lemma}
\begin{proof}
   We will prove the lemma as follows. 
 We will show for each $x \in \Z$ that 
   \begin{equation} \label{example: triple fold}
       |g(x)-g(x-1)|\leq 5.
   \end{equation}
   This means if $x,y \in \Z,$ then
   \begin{equation} \label{eq: lip max step}
       |g(x)-g(y)| = |g(\max\{x,y\}) - g(\max\{x,y\}-|x-y|)| \leq \sum_{i=1}^{|x-y|}|g(\max\{x,y\}-i+1)-g(\max\{x,y\}-i)| \leq 5|x-y|,
   \end{equation}so $g$ is $5-$Lipschitz. To this end, it suffices to show \eqref{example: triple fold} is true whenever $x \in \N.$ Indeed, considering $x \in -\N \cup\{0\},$ since $g(x) = -g(-x),$ we have that 
   $$|g(x) - g(x-1)| = |-g(-x)+g(1-x)|= |g(-x+1) -g(-x)| \leq 5,$$ so \eqref{example: triple fold} holds.
   
   Firstly, when $x = 1,$ we have that $g(1)-g(0) = 3(1)-0 = 3\leq 5.$ Now take $x \geq 2.$

   \begin{paragraph}{Case 1: $x$ and $x-1$ are both in the interval $s_{i} + [3N_i]$ for some $i \in \N$} It is clear if both $x$ and $x-1$ are in $s_{i} + I,$ where $I \in \{(0,N_i],(N_i,2N_i], (2N_i,3N_i]\},$ then $$|g(x) -g(x-1)| =3 \leq 5,$$ so \eqref{example: triple fold} holds. 
   Otherwise, $x$ and $x-1$ are in different thirds. If $x-1 \in s_{i} + (0,N_i]$ and $x \in s_{i} + (N_i,2N_i],$ then 
   $$|g(x)-g(x-1)| = |6N_i + 2 -3(N_i + 1) -3N_i| = 1\leq 5. $$
   Similarly, if $x-1 \in s_{i} + (N_i,2N_i]$ and $x \in s_{i} + (2N_i,3N_i],$ then
   $$|g(x)-g(x-1)| = |3(2N_i+1)-6N_i -2 -6N_i -2+3(2N_i)| = 1\leq 5. $$
   \end{paragraph}

   \begin{paragraph}{Case 2: $x-1 \in s_{i} + [3N_i]$ and 
   $x \in s_{i+1} + [3N_{i+1}]$ for some $i \in \N$}
      In this case, $x = s_{i+1}+1$ and $x-1 = s_i + 3N_i$. Since $s_{i+1} = s_{i}+3N_i$
   \begin{equation*}
       \begin{split}
           |g(x) - g(x-1)| &= |s_{i+1}+h^{(N_{i+1})}(1) -s_{i} - h^{(N_i)}(3N_i)| = |s_{i+1}+3-s_{i}- 3(3N_i) + 6N_i + 2|\\
           & = |3N_i -9N_i + 6N_i + 5| = 5,
       \end{split}
   \end{equation*} so \eqref{example: triple fold} holds.
   \end{paragraph}
\end{proof}
Finally, we lift the mapping $g$ from \eqref{eq: definition of f} to $\Z^d$ by defining the bijection $g^{(d)}:\Z^d \to \Z^d$ as \begin{equation} \label{eq: definition of fd}
     g^{(d)}(x_1,...,x_d) = (g(x_1),x_2,...,x_d).
 \end{equation}
\begin{lemma} \label{lemma: bilipschitz properties}
    Let $d \in \N$, $(N_i)_{i \in \N} \subseteq \N$, the bijection $g:\Z \to \Z$ given by \eqref{eq: definition of f}, dependent on $(N_i)_{i \in \N}$, and the bijection $g^{(d)}:\Z^d \to \Z^d$ given by
    \begin{equation*} \label{eq: fd}
     g^{(d)}(x_1,...,x_d) = (g(x_1),x_2,...,x_d).
 \end{equation*}
    Then
    \begin{enumerate}
    \item $g^{(d)}$ is $5-$Lipschitz;
    \item \label{enumerate: gd} $\max\{||g^{(d)}(x)-x||,||(g^{(d)})^{-1}(x)-x||\} \leq 3\max_{j \in [i]}N_j \text{ whenever $||x|| \leq 3\sum_{j=1}^iN_j$}$ for some $i \in \N;$
        \item \label{point 3, lemma 4.9} $\left(N_{i+1}(\sum_{j=1}^{i}N_j)^{-1}\right)_{i \in \N}$ is bounded if and only if $g^{(d)}$ is pointwise co-Lipschitz;
        \item \label{point 4} $(N_i)_{i \in \N}$ is bounded if and only if $g^{(d)}$ is bi-Lipschitz.
    \end{enumerate}
\end{lemma}
\begin{proof} By Lemmas \ref{lemma: lip extension higher dim} and \ref{lemma: 5-lipschitz}, $g^{(d)}$ is $5-$Lipschitz. Now recall the sequence $(s_i)_{i \in \N}$ from \eqref{eq: si} and the bijections $h^{(N_i)}:[3N_i] \to [3N_i]$ from \eqref{eq: hN} used to define $g.$ We now show that point \ref{enumerate: gd} holds. Notice that $g^{(d)}(x)$ and $x$ agree on all but the first coordinate, so it suffices to show point \ref{enumerate: gd} holds when $d=1.$ Let $i \in \N$ and $x \in \overline{B}(0,\sum_{j=1}^i 3N_j) \cap \Z.$ Point \ref{enumerate: gd} holds trivially when $x=0$. Consider the case when $x>0$. Then there exists some $k \in [i]$ such that $x \in s_k+[3N_k],$ so recalling~\eqref{eq: definition of f}
$$g(x)-x =g_+(x)-x = h^{(N_k)}(x-s_k)+s_k-x.$$
Notice that both $h^{(N_k)}(x-s_k)$ and $x-s_k$ are in $[3N_k],$ so it follows that
\begin{equation}\label{eq: gx x}
    |g(x)-x| \leq 3N_k\leq 3\max_{j \in [i]}N_j.
\end{equation}
Now consider the case when $x<0.$ Using \eqref{eq: gx x} and the fact that $g$ is an odd function 
$$|g(x)-x| = |-g(-x)+(-x)| = |g(-x)-(-x)| \leq 3\max_{j \in [i]} N_j.$$
It remains to show the same bound holds for $g^{-1}$ in place of $g.$ This follows immediately by the bijectivity of $g$ along with the fact that 
\begin{equation} \label{eq: image equality}
g\left({\overline{B}\left(0,\sum_{j=1}^i 3N_j\right)\cap\Z}\right) = \overline{B}\left(0,\sum_{j=1}^i 3N_j\right)\cap\Z\qquad  \text{ for each }i \in \N.
\end{equation}
We show \eqref{eq: image equality} is indeed true. It is clear for each $i \in \N$ that $0$ belongs to both sets in \eqref{eq: image equality} since $g(0) = 0$. Now fix $i \in \N$ and $x \in \overline{B}\left(0,\sum_{j=1}^i 3N_j\right)\cap\Z\setminus \{0\}.$ Then there exists some $k \in [i]$ such that $|x| \in s_k+[3N_k],$ so
$$|g(x)|=g(|x|) = s_k+h^{(N_k)}(|x|-s_k) \in s_k+ [3N_k] = \sum_{j=1}^{k-1}3N_j+[3N_k] \subseteq \left[\sum_{j=1}^i 3N_j \right],$$ so the left-hand set is a subset of the right-hand set in \eqref{eq: image equality}. Since $g$ is a injection, the two sets in \eqref{eq: image equality} have the same cardinality, and so the set inclusion becomes equality. Hence, we have shown that point \ref{enumerate: gd} holds.

Throughout the remainder of the proof, we use that by Proposition \ref{thm: bilip} and Lemma \ref{lemma: bilip at 0}, $g^{(d)}$ is pointwise co-Lipschitz if and only if $g$ is co-Lipschitz at $0$.
        We now show that point \ref{point 3, lemma 4.9} holds. Notice for all $i \in \N$ that $s_{i} = \sum_{j=1}^{i-1}3N_j.$ First consider the case that $\limsup_{i \to \infty}\frac{N_i}{s_{i}} = \infty$. Take the sequence $$x_i \coloneqq 2N_i + 1+s_{i}, \qquad i \in \N.$$ Since $h^{(N_i)}(2N_i+1) = 1$ for each $i \in \N$
       $$g(x_i) = g(2N_{i}+1+s_{i}) =s_{i} + h^{(N_i)}(2N_i+1)= s_{i}+1.$$ Therefore, acknowledging $g(0) = 0$ and $s_{i} \xrightarrow{i \to \infty} \infty$
   $$\liminf_{i \to \infty} \bigg|\frac{g(x_i) - g(0)}{x_i-0}\bigg|=\liminf_{i \to \infty}\bigg|\frac{s_{i}+1}{2N_i+s_{i}+1}\bigg| = \liminf_{i \to \infty} \Bigg|\frac{1+\frac{1}{s_{i}}}{\frac{2N_i}{s_{i}}+\frac{1}{s_{i}}}\Bigg|=0.$$ Thus, $g$ is not co-Lipschitz at $0,$ and so $g^{(d)}$ is nowhere co-Lipschitz.
   
   Conversely, suppose there exists $c>0$ such that for each $i \in \N_{\geq 2}$ we have $\frac{N_i}{s_{i}}\leq c.$ Let $x \in \Z \setminus \{0\}.$ Then there exists some $i \in \N$ such that $|x| \in s_{i}+[3N_i],$ and by the definition of $g$ we have that $|g(x)| \in s_{i}+[3N_i].$ Therefore,
   $$\frac{|g(x)-g(0)|}{|x-0|} \geq \frac{s_{i}+1}{s_{i}+3N_i} \geq \min\bigg\{\tfrac{1}{3N_1},\min_{i \in \N_{\geq 2}}\bigg\{\frac{1+\frac{1}{s_{i}}}{1+\frac{3N_i}{s_{i}}} \bigg\}\bigg\}\geq  \min\{\tfrac{1}{3N_1},\tfrac{1}{1+3c}\},$$ so $g$ is co-Lipschitz at $0$ and thus $g^{(d)}$ is pointwise co-Lipschitz.

 We now show that point \ref{point 4} holds. First suppose $\limsup_{i \to \infty}N_i = \infty$. For each $i \in \N_{\geq 2},$ let $x_i = s_{i} + 2N_i+1,$ and $y_i = s_{i} = s_{i-1}+3N_{i-1}.$ Notice that $$g(y_i) = g(s_{i-1}+3N_{i-1}) = s_{i-1} + h^{(N_{i-1})}(3N_{i-1}) =s_{i-1} + 3(3N_{i-1}) - 6N_{i-1} - 2 = s_{i-1} +3N_{i-1} -2 = s_{i}-2.$$ Thus
   $$\liminf_{i \to \infty}\frac{|g(x_i)-g(y_i)|}{|x_i-y_i|} = \liminf_{i \to \infty}\frac{|s_{i}+1-(s_i - 2)|}{2N_i+1} = \liminf_{i \to \infty}\frac{3}{2N_i+1} =0.$$ Thus, $g$ is not bi-Lipschitz.
   
   Conversely, assume there exists $C>0$ such that for each $i \in \N$ we have $3N_i \leq C.$ First, since $g(0) = 0,$ we trivially have that $|g(0)-0| \leq C.$ Otherwise, if $x \in \Z \setminus \{0\}$, then $|x| \in s_{i}+[3N_i]$ for some $i \in \N,$ and since $g(x) \in \text{sgn}(x)(s_{i} + [3N_i])$, we have that $|g(x)-x| \leq 3N_i \leq C.$ Thus, for each pair $x, y \in \Z$ with $x \neq y$ we have
   $$\frac{|x-y|}{|g(x)-g(y)|} \leq \frac{|g(x)-x| + |g(y)-y| + |g(x)-g(y)|}{|g(x)-g(y)|} \leq \frac{C}{1} + \frac{C}{1} + 1 = 2C+1,$$ so $g$ is bi-Lipschitz. By Lemma \ref{lemma: lip extension higher dim}, $g^{(d)}$ is bi-Lipschitz.

\end{proof}

\begin{remark}
    The construction of $g^{(d)}$ can be repeated using a `stretch and $n-$fold mapping,' for $n$ odd, in place of $h^{(N)}$ and we will emerge with a $(2n-1)-$Lipschitz bijection of $\Z^d$ which can have each of the bi-Lipschitz properties from Lemma \ref{lemma: bilipschitz properties}.
\end{remark}

\begin{example} \label{example: main}
We summarise all the possible bi-Lipschitz properties that can be achieved by the Lipschitz bijection $g^{(d)}:\Z^d \to \Z^d$ from Lemma \ref{lemma: bilipschitz properties} through the following examples. For each $i \in \N,$ if
\begin{enumerate}
  \item $N_i = 1,$ then $g^{(d)}$ is bi-Lipschitz (and is trivially pointwise co-Lipschitz)
    \item \label{enumerate: nowhere lip} $N_1 = 1$ and $N_i = 2^{i}\sum_{j=1}^{i-1}3N_j$ whenever $i \geq 2$, then $g^{(d)}$ is nowhere co-Lipschitz (and is trivially not bi-Lipschitz)
    \item \label{enumerate: pointwise non-bilip} $N_i = 2^{i},$ then $g^{(d)}$ is pointwise co-Lipschitz but not bi-Lipschitz.
\end{enumerate}
\end{example}
With this, we are ready to assemble a proof for Theorem \ref{thm:summary}.
\summary*
\begin{proof}
    Starting from the top-left and moving clockwise, the three equivalences ($\Updownarrow$) in \eqref{eq: thm summary} are given by Corollary \ref{cor: couni iff bilip}, Corollary \ref{cor: cohom equivalence} and Proposition \ref{thm: bilip}, respectively. The two forward implications ($\Longrightarrow$) in \eqref{eq: thm summary} are trivial. Moving left to right, the two reverse non-implications ($\mathrel{\rlap{\hskip .5em/}}\Longleftarrow$) of \eqref{eq: thm summary} are confirmed by points \ref{enumerate: pointwise non-bilip} and \ref{enumerate: nowhere lip} of Example \ref{example: main}, respectively.
\end{proof}
\subsection{Optimally $\omega-$co-homogeneous Lipschitz self-bijections of $\Z^d$}
By Lemma \ref{lemma: bilipschitz properties}, for a  sequence of natural numbers $(N_i)_{i \in \N}$ with $(N_{i+1}(\sum_{j=1}^{i}N_j)^{-1})_{i \in \N}$ bounded, the function $g^{(d)}$ is Lipschitz and pointwise co-Lipschitz. By Theorem \ref{thm:summary}, $g^{(d)}$ is co-homogeneous. Moreover, we have an explicit formula for a co-homogeneity function of $g^{(d)}$ from the proof of Theorem \ref{thm: co-hom equivalent to pointwise BL}, namely, define $\omega:[0,2) \to [0,\infty)$ by
$$\omega(t) \coloneqq \sup_{S>0} \sup_{\substack{u,v \in B(0,S) \cap \Z^d\\||u-v|| \leq tS}} \frac{||(g^{(d)})^{-1}(u)-(g^{(d)})^{-1}(v)||}{S}.$$
 In the next lemma, given $\omega$, we will provide a formula for a suitable sequence $(N_i)_{i \in \N}$ such that the Lipschitz bijection $g^{(d)}$ is $\omega-$co-homogeneous, with particular emphasis on the case $\omega(t) = t^{\alpha}$ for $\alpha \in (0,1],$ corresponding to McMullen's work. Moreover, we will verify that the resulting bijection admits no better co-homogeneity.

\begin{remark}
   Given $d \in \N,$ $K,X \subseteq \R^d,$ $\omega:[0,2) \to [0,\infty)$ with $\omega(0) = 0$ a non-decreasing function and $f:K \to X$ an $\omega-$homogeneous mapping, then if $\liminf_{t \to 0^+}\frac{\omega(t)}{t} =0,$ we have that $f$ is constant. Indeed, this is clear if $|K|\leq 1$, so assume $|K| \geq 2$, and fix $x, y \in K$ with $x \neq y.$ Since $\liminf_{t \to 0^+}\frac{\omega(t)}{t} =0,$ we can find a sequence $R_k \to \infty$ such that $$\frac{\omega\left(\frac{||x-y||}{R_k}\right)}{\frac{||x-y||}{R_k}} \to 0 \text{ as } k \to \infty.$$ Since $f$ is $\omega-$homogeneous, we have for each $k \in \N$ satisfying $x,y \in B(0,R_k)$ that
    $$||f(x)-f(y)|| \leq R_k\omega\left(\tfrac{||x-y||}{R_k}\right) = ||x-y|| \frac{\omega\left(\frac{||x-y||}{R_k}\right)}{\frac{||x-y||}{R_k}}$$ which goes to $0$ as $k \to \infty,$ and so $f(x) = f(y).$ Since $x$ and $y$ are arbitrary, we conclude $f$ is constant on $K.$ Thus, to construct non-constant homogeneous mappings, we will add the constraint that the function $\omega$ is concave and strictly increasing.
\end{remark}

\begin{lemma} \label{lemma: cohom examples}
Let $d \in \N$ and $\omega:[0,2) \to [0,\infty)$ be a concave, strictly increasing function with $\omega(0) = 0$. Then there exists a Lipschitz, $\omega-$co-homogeneous self-bijection of $\Z^d$ which is not $\widetilde{\omega}-$co-homogeneous for any concave, increasing function $\widetilde{\omega}$ satisfying $\widetilde{\omega}(0)=0$ and $\lim_{t \to 0^{+}}\tfrac{\widetilde{\omega}(t)}{\omega(t)}=0$. 
\end{lemma}
\begin{proof}
Without loss of generality, we may assume $\omega(1) = 1$, since a function is $\omega-$co-homogeneous if and only if it is $\frac{\omega}{\omega(1)}-$co-homogeneous. Let the sequence of natural numbers $(N_i)_{i \in \N}$ given by
$$N_i = \left\lceil \tfrac{1}{\Psi^{-1}(i)}-\tfrac{1}{\Psi^{-1}(i-1)}\right\rceil \qquad \text{with } \Psi(u)\coloneqq \int_{u}^{1}\frac{dt}{t \omega(t)}, \qquad u \in (0,1]$$
and let $g^{(d)}:\Z^d \to \Z^d$ be the Lipschitz bijection  dependent on $(N_i)_{i \in \N}$ from Lemma~\ref{lemma: bilipschitz properties}.
We will show $g^{(d)}$ is $\omega-$co-homogeneous. 
    First, we justify that $\Psi:(0,1] \to [0,\infty)$ is well-defined, invertible and surjective, and that $(N_i)_{i \in \N}$ is indeed a sequence of natural numbers. Since $\omega$ is an increasing function, we have that $\omega(t) \leq \omega(1) = 1$ for each $t \in (0,1].$ Moreover, $\omega$ is concave and thus continuous on $(0,1]$. Hence, by the comparison test for integrals $$\int_u^1 \frac{dt}{t\omega(t)} \leq \int_u^1 \frac{dt}{t} = -\log u, \qquad u \in (0,1],$$ so we have that $\Psi$ is finite and thus well-defined on $(0,1]$. Next, since $\omega(t) \leq 1$ for each $t \in (0,1],$ we have that
    $$\Psi(u)-\Psi(u+h) = \int_u^{u+h} \frac{dt}{t\omega(t)} \geq h\inf_{t \in [u,u+h]}\frac{1}{t\omega(t)} \geq h>0, \qquad u \in [0,1),\;h\in (0,1-u]$$
    so $\Psi$ is strictly decreasing, and hence $\Psi$ is invertible. We also have by the concavity of $\omega$ that $\omega(t) \geq t\omega(1)=t$ for each $t \in (0,1],$ so by the comparison test for integrals again
    $$\Psi(u) =\int_u^1 \frac{dt}{t\omega(t)} \geq \int_u^1 \frac{dt}{t^2} = \frac{1}{u}-1 \xrightarrow{u \to 0^+}\infty.$$ Hence, since $\Psi$ is differentiable and thus continuous by the fundamental theorem of calculus, then by the intermediate value theorem we have that $\Psi$ is surjective.
    Finally, since $\Psi$ is strictly decreasing, we have that $\Psi^{-1}$ is strictly decreasing too, so $N_i$ is a natural number for each $i \in \N.$

    Now we show $g^{(d)}$ is $\omega-$co-homogeneous. First, for each $i \in [0,\infty)$ we write $S_i \coloneqq \frac{1}{\Psi^{-1}(i)}.$ Fixing $i \in \N$
    \begin{equation*}
        1 = i-(i-1) = \Psi\left(\tfrac{1}{S_{i}}\right) - \Psi\left(\tfrac{1}{S_{i-1}}\right) = \int_{\frac{1}{S_i}}^{\frac{1}{S_{i-1}}}\frac{dt}{t\omega(t)}.
    \end{equation*}
    Now since $\frac{1}{t\omega(t)}$ is decreasing in $t$, and $S_i \geq S_{i-1}$ since $\Psi^{-1}$ is decreasing, we can bound the right-hand integral to obtain
    \begin{equation*}
        \left(\tfrac{1}{S_{i-1}}-\tfrac{1}{S_i}\right) \frac{1}{\frac{1}{S_i}\omega(\frac{1}{S_i})} \leq   1 \leq  \left(\tfrac{1}{S_{i-1}}-\tfrac{1}{S_i}\right) \frac{1}{\frac{1}{S_{i-1}}\omega(\frac{1}{S_{i-1}})},
    \end{equation*} so 
$$\omega\left(\tfrac{1}{S_i}\right) \geq S_i\left(\tfrac{1}{S_{i-1}}-\tfrac{1}{S_i}\right)  = \tfrac{S_i-S_{i-1}}{S_{i-1}}$$ and
$$\omega\left(\tfrac{1}{S_{i-1}}\right) \leq S_{i-1}\left(\tfrac{1}{S_{i-1}}-\tfrac{1}{S_i}\right) = \tfrac{S_i-S_{i-1}}{S_i}.$$
    Hence, we obtain 
\begin{equation}\label{eq: Si}
        N_i = \left\lceil S_i-S_{i-1}\right\rceil \leq \left\lceil S_{i-1}\omega\left(\tfrac{1}{S_i}\right)\right\rceil \leq 1+ S_{i-1}\omega\left(\tfrac{1}{S_i}\right)
    \end{equation} and
    \begin{equation}\label{eq: Si 2}
        N_i = \left\lceil S_i-S_{i-1}\right\rceil \geq  S_i-S_{i-1} \geq S_{i}\omega\left(\tfrac{1}{S_{i-1}}\right).
    \end{equation}
    \begin{claim}\label{claim: I}
       For each $i \in \N$ we have that
        $S_i \geq i.$
    \end{claim}
    \begin{proof}
Notice for each $i \in \N$ that
\begin{equation} \label{eq: eq in claim}
    \Psi\left(\tfrac{1}{i}\right) = \int_{\frac{1}{i}}^1 \frac{dt}{t\omega(t)} \leq \int_{\frac{1}{i}}^1 \frac{dt}{t^2} = i-1\leq i,
\end{equation}
 where for the first inequality we used the concavity of $\omega$ to say $\omega(t) \geq \omega(1)t = t$ for each $t \in (0,1].$ Since $\Psi$ is decreasing, applying $\Psi^{-1}$ to both sides of \eqref{eq: eq in claim} gives for each $i \in \N$ that $i\leq S_i.$
    \end{proof}
    By Claim \ref{claim: I}, for each $i\in \N$ we have that
  \begin{equation} \label{eq: sum Nj condition 2}
        \sum_{j=1}^iN_j = \sum_{j=1}^i\lceil S_j-S_{j-1}\rceil \leq \sum_{j=1}^i (S_j-S_{j-1}+1) = S_i -S_0+i\leq S_i+i \leq 2S_i.
 \end{equation} Similarly, since $S_0 = 1,$ notice
 \begin{equation} \label{eq: sum Nj condition prime}
        \sum_{j=1}^iN_j = \sum_{j=1}^i\lceil S_j-S_{j-1}\rceil \geq \sum_{j=1}^i S_j-S_{j-1} = S_i -S_0
 = S_i -1.
 \end{equation} Note that the last expression of \eqref{eq: sum Nj condition prime} is at least $\frac{S_i}{2}$ whenever $S_i \geq 2.$ By Claim \ref{claim: I}, $S_i \geq i \geq 2$ for each $i \geq 2.$ If $i = 1$ and $S_1 =1,$ the left-hand expression of \eqref{eq: sum Nj condition prime} gives $N_1$ which is at least $1$ by definition, and so $N_1 \geq \frac{S_1}{2}.$ Together, we obtain for each $i \in \N$ that
 \begin{equation} \label{eq: sum Nj condition}
     \sum_{j=1}^i N_j \geq \tfrac{S_i}{2}.
 \end{equation}
  Now fix $R>0.$ We have that $R \in (3\sum_{j=1}^{i-1}N_j,3\sum_{j=1}^{i}N_j]$ for some $i \in \N.$ Thus, using \eqref{eq: Si}, \eqref{eq: sum Nj condition 2} and \eqref{eq: sum Nj condition}
    \begin{equation} \label{eq: 4max}
        \begin{split}
            N_i &\leq 1+S_{i-1}\omega\left(\tfrac{1}{S_i}\right) \leq 1+2\left(\sum_{j=1}^{i-1}N_j\right)\omega\left(\frac{2}{\sum_{j=1}^iN_j}\right)
             \\&\leq 1+\tfrac{2R}{3}\omega\left(\tfrac{6}{R}\right) \leq R\omega\left(\tfrac{1}{R}\right) +\tfrac{2R}{3}6\omega\left(\tfrac{1}{R}\right) = 5R\omega\left(\tfrac{1}{R}\right),
        \end{split}
    \end{equation} where in the fourth inequality we used the concavity of $\omega$ to show ${\omega(\frac{6}{R}) \leq 6\omega(\frac{1}{R})}$ and $1=\omega(1) \leq R\omega(\frac{1}{R}).$
    Now we need the following claim.
    \begin{claim} \label{claim: nondec}
        $(N_i)_{i \in \N}$ is a non-decreasing sequence.
    \end{claim}
    \begin{proof}
    First, we claim that $\frac{\omega(t)}{t}$ is a non-increasing function on $(0,1]$. Let $t_1,t_2 \in (0,1]$ with $t_1\leq t_2$. Since $\omega$ is concave and $\omega(0) = 0$
        $$\omega(t_1) =\omega\left(\left(1-\tfrac{t_1}{t_2}\right)0+\tfrac{t_1}{t_2} t_2\right) \geq \left(1-\tfrac{t_1}{t_2}\right)\omega(0) + \tfrac{t_1}{t_2}\omega(t_2) = \tfrac{t_1}{t_2}\omega(t_2),$$
        so $\frac{\omega(t_1)}{t_1} \geq \frac{\omega(t_2)}{t_2},$ as required.
        Now by the fundamental theorem of calculus, we have that
        $$\Psi'(u) = -\tfrac{1}{u\omega (u)}, \qquad u \in (0,1].$$ Hence, for each $i \in [0,\infty)$ we have that
        \begin{equation*}
            (\Psi^{-1}(i))' = \tfrac{1}{\Psi'(\Psi^{-1}(i))} = -\Psi^{-1}(i)\omega(\Psi^{-1}(i)).
        \end{equation*}
        Thus, for each $i \in [0,\infty)$
        \begin{equation*}
        \begin{split}
            S_i' &= \left(\tfrac{1}{\Psi^{-1}(i)}\right)' = -\tfrac{1}{(\Psi^{-1}(i))^2} \left(\Psi^{-1}(i)\right)' = -\tfrac{1}{(\Psi^{-1}(i))^2} \cdot-\Psi^{-1}(i)\omega(\Psi^{-1}(i)) = \tfrac{\omega(\Psi^{-1}(i))}{\Psi^{-1}(i)}.
            \end{split}
        \end{equation*} Since both $\Psi^{-1}$ and $\frac{\omega(t)}{t}$ are non-increasing functions, we have that $S_i'$ is a non-decreasing function, and so $S_i$ is convex.
        Therefore, for each $i \in \N$
        $$2S_i = 2S_{\frac{1}{2}(i-1)+\frac{1}{2}(i+1)} \leq 2\left({\tfrac{1}{2}S_{i-1} +\tfrac{1}{2}S_{i+1}}\right)=S_{i-1}+S_{i+1}$$
        so we have that
        $$S_{i+1}-S_i \geq S_i - S_{i-1} \qquad \text{for each } i \in \N.$$ 
        Taking the ceiling of both sides, we obtain that $N_{i+1} \geq N_i$ for each $i \in \N,$ as required.
    \end{proof}
    With the insertion of Claim \ref{claim: nondec}, point \ref{enumerate: gd} of Lemma \ref{lemma: 5-lipschitz} gives that
    $$||(g^{(d)})^{-1}(x)-x|| \leq 3N_i \qquad \text{whenever $||x|| \leq 3\sum_{j=1}^iN_j$}.$$
    Then fixing $x,y \in B(0,R) \cap \Z^d$ with $x \neq y,$ we have that
    \begin{equation*}
    \begin{split}
        ||(g^{(d)})^{-1}(x) - (g^{(d)})^{-1}(y)|| &\leq ||(g^{(d)})^{-1}(x)-x|| + ||(g^{(d)})^{-1}(y)-y|| + ||x-y|| \leq 6N_i + ||x-y||\\
        &\leq 30R\omega\left(\tfrac{1}{R}\right)+R\omega\left(\tfrac{||x-y||}{R}\right)\leq 31R\omega\left(\tfrac{||x-y||}{R}\right),
         \end{split}
    \end{equation*} where for the penultimate inequality we used \eqref{eq: 4max} and that $\omega(t) \geq t$ for each $t \in (0,1]$ by the concavity of $\omega$, and for the last inequality we used that $||x-y|| \geq 1.$ Thus, $g^{(d)}$ is $\omega-$co-homogeneous.

    It remains to check $g^{(d)}$ does not witness any `better' co-homogeneity. Fix a concave, strictly increasing function $\widetilde{\omega}:[0,2) \to [0,\infty)$ with $\widetilde{\omega}(0) = 0$ and $\lim_{t \to 0^+} \frac{\widetilde{\omega}(t)}{\omega(t)} = 0$. 
    Define the sequences $(R_i)_{i \in \N}\subseteq \N,$ and $(x_i)_{i \in \N},(y_i)_{i \in \N} \subseteq \N^d$ as
    $$R_i = 3\sum_{j=1}^{i+1}N_j, \qquad x_i = \left(1+3\sum_{j=1}^{i}N_j,0,...,0\right),\qquad y_i =\left(-2+3\sum_{j=1}^{i}N_j,0,...,0\right) .$$
  Notice for each $i \in \N$ that $x_i,y_i \in B(0,R_i) \cap \Z^d.$ 
  By \eqref{eq: sum Nj condition 2}, \eqref{eq: sum Nj condition} and the fact that $S_{i+1} \geq S_i,$ we have that $$\tfrac{3S_{i}}{2} \leq \tfrac{3S_{i+1}}{2} \leq R_i \leq 6S_{i+1}, \qquad i \in \N.$$ Together with \eqref{eq: Si 2}, we have that
  \begin{equation} \label{eq: niriomegai}
\frac{N_{i+1}}{R_i\omega\left(\frac{1}{R_i}\right)} \geq \frac{N_{i+1}}{6S_{i+1}\omega\left(\frac{2}{3S_{i}}\right)}\geq \frac{N_{i+1}}{6S_{i+1}\omega\left(\frac{1}{S_{i}}\right)}\geq \frac{1}{6}, \qquad i \in \N. \end{equation}
 Noting by the definition of $g^{(d)}$ that $(g^{(d)})^{-1}(x_i) = x_i+(2N_{i+1},0,...,0)$ and $(g^{(d)})^{-1}(y_i) = y_i + (2,0,...,0) = x_i-(1,0,...,0),$ we have that
  \begin{equation*}
      \begin{split}
          \frac{||(g^{(d)})^{-1}(x_i) - (g^{(d)})^{-1}(y_i)||}{R_i \widetilde{\omega}\left(\frac{1}{R_i}\right)}  &= \frac{2N_{i+1}+1}{R_i \widetilde{\omega}\left(\frac{1}{R_i}\right)} \geq \frac{N_{i+1}}{R_i \widetilde{\omega}\left(\frac{1}{R_i}\right)}
          \geq \frac{1}{6}\frac{{\omega}\left(\frac{1}{R_i}\right)}{\widetilde{\omega}\left(\frac{1}{R_i}\right)}\xrightarrow{i \to \infty}\infty,
      \end{split}
  \end{equation*} where for the final inequality we used \eqref{eq: niriomegai}. Hence, $g^{(d)}$ is not $\widetilde{\omega}-$co-homogeneous.
\end{proof}
    \begin{example}
        If $\alpha \in (0,1]$, using the formula from the proof of Lemma \ref{lemma: cohom examples}, the sequence $(N_i)_{i \in \N}$ given by $$N_i\coloneqq \lceil(\alpha i+1)^{\frac{1}{\alpha}}-(\alpha (i-1)+1)^{\frac{1}{\alpha}} \rceil \asymp i^{\frac{1}{\alpha}-1}, \qquad i \in \N$$ witnesses that the Lipschitz bijection $g^{(d)}$ is co-$\alpha-$H\"{o}lder homogeneous. Moreover, $g^{(d)}$ is not co-$\tilde{\alpha}-$H\"{o}lder homogeneous for any $\tilde{\alpha} >\alpha.$
    \end{example}

    \paragraph{Acknowledgements}
    AB thanks the University of Birmingham for financial support.

\begin{flushleft}
    Michael Dymond,\\
    School of Mathematics, University of Birmingham, Birmingham, B15 2TT, United Kingdom.\\
    \href{mailto:m.dymond@bham.ac.uk}{m.dymond@bham.ac.uk} 
\end{flushleft}

\begin{flushleft}
    Ashwin Bhat,\\
    School of Mathematics, University of Birmingham, Birmingham, B15 2TT, United Kingdom.\\
    \href{mailto:axb1805@student.bham.ac.uk}{axb1805@student.bham.ac.uk} 
\end{flushleft}
\end{document}

%% file: macros.tex
\usepackage{amsmath}
\usepackage{amssymb}
\usepackage{enumerate}
\usepackage{amsthm}
\usepackage{todonotes}
\usepackage{mathtools} 
\usepackage[UKenglish]{babel}
\usepackage[T1]{fontenc}
\usepackage[utf8]{inputenc}
\usepackage{lmodern} 
\usepackage{tikz}
\usepackage{thmtools} 
\usepackage{thm-restate} 
\usepackage{stackrel} 
\usepackage{hyperref}
\usepackage[left = 23mm,right=23mm, top = 23mm,bottom =23mm, paper = a4paper]{geometry}
\usetikzlibrary {arrows.meta}
\usetikzlibrary{decorations.pathreplacing}
\usepgflibrary {shadings}

\declaretheorem[name=Theorem,numberwithin=section]{thm} 
\newtheorem*{thm*}{Theorem}

\newtheorem*{define*}{Definition}
\newtheorem{define}[thm]{Definition}

\newtheorem*{lemma*}{Lemma}
\newtheorem{lemma}[define]{Lemma}

\newtheorem*{algorithm*}{Algorithm}

\newtheorem*{construction*}{Construction}

\newtheorem*{prop*}{Proposition}
\newtheorem{prop}[define]{Proposition}

\newtheorem*{obs*}{Observation}

\newtheorem*{fact*}{Fact}

\newtheorem*{remark*}{Remark}
\newtheorem{remark}[define]{Remark}

\newtheorem*{claim*}{Claim}
\newtheorem{claim}[define]{Claim}

\newtheorem*{quest*}{Question}

\newtheorem*{cor*}{Corollary}
\newtheorem{cor}[define]{Corollary}

\newtheorem*{conjecture*}{Conjecture}

\newtheorem*{question*}{Question}
\newtheorem{question}[define]{Question}

\newtheorem*{example*}{Example}
\newtheorem{example}[define]{Example}

\usepackage[mathlines]{lineno}
\usepackage[normalem]{ulem}
\usepackage{graphicx}

\usepackage{color}
\definecolor{grey}{rgb}{.7,.7,.7}
\definecolor{blue}{rgb}{0,0,.8}
\definecolor{red}{rgb}{.8,0,0}
\definecolor{green}{rgb}{0,.4,0}
\definecolor{gold}{rgb}{0.8,0.6,0.1}
\definecolor{brown}{rgb}{0.8,0.4,0.1}
\definecolor{arxivcolor}{rgb}{0.5,0.5,0}
\definecolor{journalcolor}{rgb}{0.5,0,1}
\definecolor{purple}{rgb}{0.6,0.2,0.6}
\definecolor{pastelgreen}{rgb}{0,0.65,0.1}

\long\def\del#1{\color{blue}\ifmmode\text{\sout{\ensuremath{#1}}}\else\sout{#1}\fi\normalcolor}

\newcommand{\R}{\mathbb{R}}

\newcommand{\Z}{\mathbb{Z}}
\newcommand{\N}{\mathbb{N}}

\newcommand{\leb}{\mc{L}}

\newcommand{\bilip}{\operatorname{bilip}}
\DeclareMathOperator{\diam}{diam}

\newcommand{\infnorm}[1]{\left\|#1\right\|_\infty}

\newcommand{\mc}[1]{\mathcal{#1}}

